\documentclass[final]{siamltex1213}

\usepackage{xcolor,colortbl}

\usepackage{tikz}
\usetikzlibrary{fit}\usetikzlibrary{arrows,shapes,positioning}
\usetikzlibrary{decorations.markings}
\tikzstyle arrowstyle=[scale=1]
\tikzstyle directed=[postaction={decorate,decoration={markings,
    mark=at position .5 with {\arrow[arrowstyle]{stealth}}}}]
\tikzstyle reverse directed=[postaction={decorate,decoration={markings,
    mark=at position .5 with {\arrowreversed[arrowstyle]{stealth};}}}]

\newtheorem{assump}{Assumption}
\newtheorem{remark}{Remark}

\usepackage{mathtools}
\usepackage[normalem]{ulem} 

\newcommand{\vN}{{\mathbf{N}}}

\newcommand{\vR}{{\mathbf{R}}}

\newcommand{\cH}{{\mathcal{H}}}

\newcommand{\dom}{{\mathrm{dom}}} 
\newcommand{\range}{{\mathrm{range}}} 

\newcommand{\prox}{\mathbf{prox}}

\newcommand{\tnabla}{\widetilde{\nabla}}

\newcommand{\TFDRS}{T_{\mathrm{FDRS}}}

\newcommand{\alphaFDRS}{\alpha_{\mathrm{FDRS}}}
\newcommand{\alphaFDRSV}{\alpha_{\mathrm{FDRS}}^V}

\DeclareMathOperator*{\argmin}{arg\,min}

\DeclareMathOperator*{\Min}{minimize}

\DeclareMathOperator*{\zer}{zer}

\newcommand\restr[2]{{
  \left.\kern-\nulldelimiterspace 
  #1 
  \vphantom{\big|} 
  \right|_{#2} 
  }}

\newcommand{\bc}{\begin{center}}
\newcommand{\ec}{\end{center}}

\newcommand{\bdm}{\begin{displaymath}}
\newcommand{\edm}{\end{displaymath}}

\newcommand{\beq}{\begin{equation}}
\newcommand{\eeq}{\end{equation}}

\newcommand{\bfl}{\begin{flushleft}}
\newcommand{\efl}{\end{flushleft}}

\newcommand{\bt}{\begin{tabbing}}
\newcommand{\et}{\end{tabbing}}

\newcommand{\beqn}{\begin{align}}
\newcommand{\eeqn}{\end{align}}

\newcommand{\beqs}{\begin{align*}} 
\newcommand{\eeqs}{\end{align*}}  

\usepackage[lined,boxed,commentsnumbered, ruled,vlined]{algorithm2e}

\newcommand\numberthis{\addtocounter{equation}{1}\tag{\theequation}}

\DeclarePairedDelimiter{\dotp}{\langle}{\rangle}

\newcommand{\refl}{\mathbf{refl}}

\usepackage{booktabs}

\def\cut#1{{}}

\title{Convergence rate analysis of the forward-Douglas-Rachford splitting scheme\thanks{This work is partially supported by grants NSF DGE-0707424 (graduate research fellowship program) and NSF DMS-1317602. 
}}

\author{Damek Davis\thanks{Department of Mathematics, University of California, Los Angeles
              Los Angeles, CA 90025, USA
              \email{damek@math.ucla.edu}}}

\begin{document}
\maketitle
\slugger{siopt}{xxxx}{xx}{x}{x--x}

\begin{abstract}
Operator splitting schemes are a class of powerful algorithms that solve complicated monotone inclusion and convex optimization problems that are built from many simpler pieces. They give rise to algorithms in which all simple pieces of the decomposition are processed individually. This leads to easily implementable and highly parallelizable or distributed algorithms, which often obtain nearly state-of-the-art performance.

In this paper, we analyze the convergence rate of the forward-Douglas-Rachford splitting (FDRS) algorithm, which is a generalization of the forward-backward splitting (FBS) and Douglas-Rachford splitting (DRS) algorithms. Under general convexity assumptions, we derive the ergodic and nonergodic convergence rates of the FDRS algorithm, and show that these rates are the best possible. Under Lipschitz differentiability assumptions, we show that the best iterate of FDRS converges as quickly as the last iterate of the FBS algorithm. Under strong convexity assumptions, we derive convergence rates for a sequence that strongly converges to a minimizer. Under strong convexity and Lipschitz differentiability assumptions, we show that FDRS converges linearly. We also provide examples where the objective is strongly convex, yet FDRS converges arbitrarily slowly. Finally, we relate the FDRS algorithm to a primal-dual forward-backward splitting scheme and clarify its place among existing splitting methods. Our results show that the FDRS algorithm automatically adapts to the regularity of the objective functions and achieves rates that improve upon the sharp worst case rates that hold in the absence of smoothness and strong convexity.
\end{abstract}



\pagestyle{myheadings}
\thispagestyle{plain}
\markboth{D. Davis}{Convergence rates for forward-Douglas-Rachford}

\section{Introduction}\label{sec:intro}
Operator-splitting schemes are algorithms for splitting complicated problems arising in PDE, monotone inclusions, optimization, and control into many simpler subproblems.  The achieved decomposition can give rise to inherently parallel and, in some cases, distributed algorithms. These characteristics are particularly desirable for large-scale problems that arise in machine learning, finance, control, image processing, and PDE~\cite{boyd2011distributed}.

In optimization, the Douglas-Rachford splitting (DRS) algorithm~\cite{lions1979splitting} minimizes sums of (possibly) nonsmooth functions $f, g : \cH \rightarrow (-\infty, \infty]$ on a Hilbert space $\cH$:
\begin{align}\label{eq:DRSprob}
\Min_{x \in \cH} f(x) + g(x).
\end{align}
During each step of the algorithm, DRS applies the proximal operator, which is the basic subproblem in nonsmooth minimization, to $f$ and $g$ individually rather than to the sum $f+g$. Thus, the key assumption in DRS is that $f$ and $g$ are easy to minimize \emph{independently}, but the sum $f+g$ is difficult to minimize. We note that many complex objectives arising in machine learning~\cite{boyd2011distributed} and signal processing~\cite{combettes2011proximal} are the sum of nonsmooth terms with simple or closed-form proximal operators. 

The forward-backward splitting (FBS) algorithm~\cite{passty1979ergodic} is another technique for solving~\eqref{eq:DRSprob} when $g$ is known to be \emph{smooth}.  In this case, the proximal operator of $g$ is never evaluated. Instead, FBS combines gradient (forward) steps with respect to $g$ and proximal (backward) steps with respect to $f$.  FBS is especially useful when the proximal operator of $g$ is complex and its gradient is simple to compute.

Recently, the forward-Douglas-Rachford splitting (FDRS) algorithm~\cite{briceno2012forward} was proposed to combine DRS and FBS and extend their applicability (see Algorithm~\ref{alg:FDRS}). More specifically, let $V \subseteq \cH$ be a \emph{closed vector space} and suppose $g$ is smooth. Then FDRS applies to the following constrained problem:
\begin{align}\label{eq:FDRSprob}
\Min_{x \in V} f(x) + g(x).
\end{align}
Throughout the course of the algorithm, the proximal operator of $f$, the gradient of $g$, and the projection operator onto $V$ are all employed separately. 

The FDRS algorithm can also apply to affinely constrained problems. Indeed, if $V = V_0 + b$ for a closed vector subspace $V_0 \subseteq \cH$ and a vector $b \in \cH$, then Problem~\eqref{eq:FDRSprob} can be reformulated as
\begin{align}\label{eq:FDRSprobaffine}
\Min_{x \in V_0} f(x+b) + g(x+b).
\end{align}
For simplicity, we only consider linearly constrained problems.

The FDRS algorithm is a generalization of the generalized forward-backward splitting  (GFBS) algorithm~\cite{raguet2013generalized}, which solves the problem $\Min_{x \in \cH} \sum_{i=1}^n f_i(x) + g(x)$ where $f_i : \cH \rightarrow (-\infty, \infty]$ are closed, proper, convex and (possibly) nonsmooth. In the GFBS algorithm, the proximal mapping of each function $f_i$ is evaluated \emph{in parallel}. We note that GFBS can be derived as an application of FDRS to the equivalent problem:
\begin{align}\label{eq:GFBSinFDRS}
\min_{\substack{(x_1, x_2, \ldots, x_n) \in \cH^n \\ x_1=x_2= \cdots= x_n}} \sum_{i=1}^n f_i(x_i) + g \left( \frac{1}{n}\sum_{i=1}^n x_i\right).
\end{align}
In this case, the vector space $V= \{(x, \ldots, x) \in \cH^n \mid x \in \cH\}$ is the diagonal set of $\cH^n$ and the function $f$ is separable in the components of $(x_1, \cdots, x_n)$\cut{: $f(x_1, \ldots, x_n) = \sum_{i=1}^n f_i(x_i)$}.

The FDRS algorithm is the only primal operator-splitting method capable of using all structure in Equation~\eqref{eq:FDRSprob}. 
In order to achieve good practical performance, the other primal splitting methods require stringent assumptions on $f, g,$ and $V$. Primal DRS cannot use the smooth structure of $g$, so the proximal operator of $g$ must be simple. On the other hand, primal FBS and forward-backward-forward splitting (FBFS)~\cite{tseng2000modified} cannot separate the coupled nonsmooth structure of $f$ and $V$, so minimizing $f(x)$ subject to $x \in V$ must be simple. In contrast, FDRS  achieves good practical performance if it is simple to minimize $f$, evaluate $\nabla g$, and project onto $V$.

Modern primal-dual splitting methods~\cite{chambolle2011first,esser2010general,condat2013primal,vu2013splitting,briceno2011monotone+,komodakis2014playing} can also decompose problem~\eqref{eq:FDRSprob}, but they introduce extra variables and are, thus, less memory efficient. It is unclear whether FDRS will perform better than primal-dual methods when memory is not a concern. However, it is easier to choose algorithm parameters for FDRS and, hence, it can be more convenient to use in practice. 

\subsection*{Application: constrained quadratic programming and support vector machines} Let $d$ and $m$ be natural numbers. Suppose that $Q \in \vR^{d \times d}$ is a symmetric positive semi-definite matrix, $c \in \vR^d$ is a vector, $C \subseteq \vR^d$ is a constraint set,  $A \in \vR^{m \times d}$ is a linear map, and $b \in \vR^{m}$ is a vector.  Consider the  problem:
\begin{align*}
\Min_{x\in \vR^d} & \; \frac{1}{2}\dotp{ Qx, x} + \dotp{ c, x} \numberthis \label{eq:cqprogramming}\\
\text{subject to:} &\;  x \in C \\
& \; Ax = b.
\end{align*}
Problem~\eqref{eq:cqprogramming} arises in the dual form soft-margin kernelized support vector machine classifier~\cite{cortes1995support} in which $C$ is a box constraint, $b$ is $0$, and $A$ has rank one. Note that by the argument in~\eqref{eq:FDRSprobaffine}, we can always assume that $b = 0$. 

Define the smooth function $g(x) := (1/2)\dotp{ Qx, x} + \dotp{ c, x}$, the indicator function $f(x) := \chi_{C}(x)$ (which is $0$ on $C$ and $\infty$ elsewhere), and the vector space $V := \{x \in \vR^d \mid Ax = 0\}$. With this notation,~\eqref{eq:cqprogramming} is in the form~\eqref{eq:FDRSprob} and, thus, FDRS can be applied. This splitting is nice because $\nabla g(x) = Qx + c$ is simple whereas the proximal operator of $g$ requires a matrix inversion $\prox_{\gamma g} = (I_{\vR^d} + \gamma Q)^{-1}\circ (I_{\vR^d} - \gamma c),$ which is expensive for large-scale problems. 

\subsection{Goals, challenges, and approaches}

This work seeks to characterize the convergence rate of the FDRS algorithm applied to Problem~\eqref{eq:FDRSprob}.  Recently, \cite{davis2014convergence} has shown that the sharp convergence rate of the fixed-point residual (FPR) (see Equation~\eqref{eq:FPR}) of the FDRS algorithm is $o(1/(k+1))$ . To the best of our knowledge, nothing is else is known about the convergence rate of FDRS. Furthermore,   it is unclear how the FDRS algorithm relates to other algorithms. We seek to fill this gap. 

The techniques used in this paper are based on~\cite{davis2014convergencepd,davis2014convergence,davis2014convergencefast}. These techniques are quite different from those used in classical objective error convergence rate analysis. The classical techniques do not apply because the FDRS algorithm is driven by the fixed-point iteration of a nonexpansive operator, not by the minimization of a model function\cut{, such as in the FBS algorithm~\cite{beck2009fast}}.  Thus, we must explicitly use the properties of nonexpansive operators in order to derive convergence rates for the objective error.

We summarize our contributions  and techniques as follows:
\begin{romannum}
\item We analyze the objective error convergence rates (Theorems~\ref{thm:FDRSergodic} and~\ref{thm:FDRSnonergodic}) of the FDRS algorithm under general convexity assumptions. We show that FDRS is, in the worst case,
\emph{nearly as slow as the subgradient method} yet \emph{nearly as fast as the proximal point algorithm (PPA) in the ergodic sense}.  Our nonergodic rates are shown by relating the objective error to the FPR through a \emph{fundamental inequality}.  We also show that the derived rates are sharp through counterexamples (Remarks~\ref{rem:optimalergodic} and~\ref{rem:optimalnonergodic}).
\item We show that if $f$ or $g$ is strongly convex, then a natural sequence of points converges strongly to a minimizer. Furthermore, the
\emph{best iterate} converges with rate $o(1/(k + 1))$, the \emph{ergodic iterate} converges with rate $O(1/(k + 1))$, and the \emph{nonergodic iterate} converges with rate $o(1/\sqrt{k+1})$.  The results follow by showing that a certain sequence of squared norms is summable.  We also show that some of the derived rates are sharp by constructing a novel counterexample (Theorem~\ref{eq:nonlinearconvergence}).

\item We show that if $f$ is differentiable and $\nabla f$ is Lipschitz, then the \emph{best iterate} of the FDRS algorithm has objective error of order $o(1/(k+1))$ (Theorem~\ref{thm:lipschitzbest}). This rate is an improvement over the sharp $o(1/\sqrt{k+1})$ convergence rate for nonsmooth $f$. The result follows by showing that the objective error is summable.

\item We establish scenarios under which FDRS converges linearly (Theorem~\ref{thm:linearconvergence}) and show that linear convergence is impossible under other scenarios (Theorem~\ref{eq:nonlinearconvergence}). 

\item We show that even if $f$ and $g$ are strongly convex, the FDRS algorithm can converge \emph{arbitrarily slowly} (Theorem~\ref{thm:arbitrarilyslow}). 

\item We show that the FDRS algorithm is the limiting case of a recently developed primal-dual forward-backward splitting algorithm (Section~\ref{sec:PD}) and, thus, clarify how FDRS relates to existing algorithms. 
\end{romannum}

Our analysis builds on the techniques and results of~\cite{briceno2012forward,davis2014convergence,davis2014convergencefast}. The rest of this section contains a brief review of these results.

\subsection{Notation and facts}\label{sec:notation}

Most of the definitions and notation that we use in this paper are standard and can be found in~\cite{bauschke2011convex}. Throughout this paper, we use $\cH$ to denote (a possibly infinite dimensional) Hilbert space. In fixed-point iterations, $(\lambda_j)_{j \geq 0} \subset \vR_+$ will denote a sequence of relaxation parameters, and 
\begin{equation}\label{def:Lambda}\Lambda_k := \sum_{i=0}^k \lambda_i
\end{equation} is its $k$th partial sum.\cut{ To ease notational memory, the reader may assume that $\lambda_k \equiv 1$ and $\Lambda_k =(k+1)$.}\cut{ Given any sequence  $(x^j)_{j \geq 0}\subset \cH$, we let 
\begin{align}\label{eq:ergiterate}
\overline{x}^k = ({1}/{\Lambda_k})\sum_{i=0}^k \lambda_i x^i
\end{align}
denote its $k$th average  with respect to the sequence $(\lambda_j)_{j \geq 0}$.  We call $(\overline{x}^j)_{j \geq 0}$ the \emph{ergodic} sequence and we call $(x^j)_{j \geq 0}$ the \emph{nonergodic} sequence.}

For any subset $C \subseteq \cH$, we define the distance function:
\begin{align}\label{eq:distancefunction}
d_{C}(x) :=  \inf_{y \in C} \|x -y\|. 
\end{align}
In addition, we define the indicator function $\chi_{C} : \cH \rightarrow \{0, \infty\}$ of $C$: for all $x \in C$ and $y \in \cH \backslash C$, we have $\chi_{C}(x) = 0$ and $\chi_C(y) = \infty$. 

Given a closed, proper, and convex function $f : \cH \rightarrow (-\infty, \infty]$, the set $\partial f(x) = \{p \in \cH \mid \text{for all } y\in \cH, f(y) \geq f(x) + \dotp{y-x, p}\}$ denotes its subdifferential at $x$ and
\begin{align*}
\tnabla f(x) \in \partial f(x)
\end{align*}
denotes a subgradient. (This notation was used in \cite[Eq. (1.10)]{bertsekas2011incremental}.) If $f$ is G\^{a}teaux differentiable at $x \in \cH$, we have $\partial f(x) = \{\nabla f(x)\}$~\cite[Proposition 17.26]{bauschke2011convex}.

Let $I_{\cH}: \cH \rightarrow \cH$ be the identity map on $\cH$. For any $x \in \cH$ and $\gamma \in \vR_{++}$, we let 
\begin{align*}
\prox_{\gamma f}(x) := \argmin_{y \in \cH} \left(f(y) + \frac{1}{2\gamma} \|y - x\|^2\right) \quad \mathrm{and} \quad \refl_{\gamma f} := 2\prox_{\gamma f} - I_{\cH},
\end{align*}
which are known as the \emph{proximal} and \emph{reflection} operators, respectively. 

 The subdifferential of the indicator function $\chi_V$ where $V \subseteq \cH$ is a closed vector subspace is defined as follows: for all $x \in \cH$,
\begin{align}\label{eq:normalcone}
\partial \chi_V(x) = 
 \begin{cases} 
V^\perp & \text{if } x \in \cH;\\
\emptyset &\text{otherwise}
\end{cases}
\end{align}
where $V^\perp$ is the orthogonal complement of $V$. Evidently, if $P_V(\cdot) = \argmin_{y \in V} \|y  - \cdot \|^2$ is the projection onto $V$, then
\begin{align*}
\prox_{\gamma \chi_V} = P_V   && \mathrm{and} && \refl_{\gamma \chi_V} = 2P_V - I_{\cH} = P_V - P_{V^\perp},
\end{align*}
and these operators are independent of $\gamma$.

Let $\lambda > 0$, let $L \geq 0$, and let $T : \cH \rightarrow \cH$ be a map. The map $T$ is called  \emph{$L$-Lipschitz} continuous if $\|Tx - Ty\| \leq L\|x-y\|$ for all $x, y \in \cH$. The map $T$ is called \emph{nonexpansive} if it is $1$-Lipschitz.  We also use the notation:
\begin{align}\label{eq:averagednotation}
T_{\lambda} := (1-\lambda)I_{\cH} + \lambda T.
\end{align}
If $\lambda \in (0, 1)$ and $T$ is nonexpansive, then $T_{\lambda}$ is called \emph{$\lambda$-averaged}~\cite[Definition 4.23]{bauschke2011convex}.

We call the following identity the \emph{cosine rule}:
\begin{align*}
\|y-z\|^2+2\dotp{y-x,z-x}=\|y-x\|^2+\|z-x\|^2,\quad\forall x,y,z\in\cH \numberthis\label{eq:cosinerule}.
\end{align*}
Young's inequality is the following: for all $a, b \geq 0$ and $\varepsilon > 0$, we have 
\begin{align}\label{eq:young}
ab \leq  a^2/(2\varepsilon) + \varepsilon b^2/2.
\end{align}

\subsection{Assumptions}
~\\
\begin{assump}[Convexity]
$f$ and $g$ are closed, proper, and convex. 
\end{assump}

We also assume the existence of a particular solution to~\eqref{eq:FDRSprob}

\begin{assump}[Solution existence]
 $\zer(\partial f+ \nabla g + \partial \chi_V) \neq \emptyset$
\end{assump}

Finally we assume that $\nabla g$ is sufficiently nice.
\begin{assump}[Differentiability]\label{assump:lipschitz}
The function $g$ is differentiable,  $\nabla g$ is $(1/\beta)$-Lipschitz, and $P_V \circ \nabla g \circ  P_V$ is $(1/\beta_{V})$-Lipschitz.
\end{assump}

\subsection{The FDRS algorithm}

FDRS is summarized in Algorithm~\ref{alg:FDRS}.

\begin{algorithm}[H]
\SetKwInOut{Input}{input}\SetKwInOut{Output}{output}
\SetKwComment{Comment}{}{}
\BlankLine
\Input{$z^0 \in \cH, ~\gamma \in (0, \infty) , ~(\lambda_j)_{j \geq 0} \in (0, \infty)$}
\For{$k=0,~1,\ldots$}{
$z^{k+1} = (1-\lambda_k)z^k +   \lambda_k\left(\frac{1}{2}I_{\cH} + \frac{1}{2}\refl_{\gamma f} \circ \refl_{\chi_V}\right) \circ (I - \gamma P_V \circ \nabla g \circ P_V)(z^k) $\;}
\caption{{Relaxed Forward-Douglas-Rachford splitting (relaxed FDRS)}}
\label{alg:FDRS}
\end{algorithm}
For now, we do not specify the stepsize parameters. See section~\ref{sec:convergence} for choices that ensure convergence and, see Lemma~\ref{lem:FDRSidentities} and Figure~\ref{fig:FDRSsquare} for intuition. 

Evidently, Algorithm~\ref{alg:FDRS} has the form: for all $k \geq 0$, $z^{k+1} = (\TFDRS)_{\lambda_k}(z^k)$ where
\begin{align}\label{eq:TFDRS}
\TFDRS &:= \left(\frac{1}{2}I_{\cH} + \frac{1}{2}\refl_{\gamma f} \circ \refl_{  \chi_V}\right) \circ (I_\cH - \gamma P_V \circ \nabla g \circ P_V).
\end{align}
Because $\TFDRS$ is nonexpansive (Part~\ref{prop:basicprox:part:alphaFDRS} of Proposition~\ref{prop:basicprox}), it follows that the FDRS algorithm is a special case of the Krasnosel'ski\u{\i}-Mann (KM) iteration~\cite{krasnosel1955two,mann1953mean,combettes2004solving}.

By choosing particular $f, g$ and $V$, we recover several other splitting algorithms:
\begin{align*}
\text{DRS:} \; (g \equiv 0) \quad z^{k+1} &= (1-\lambda_k)z^k +   \lambda_k\left(\frac{1}{2}I_{\cH} + \frac{1}{2}\refl_{\gamma f} \circ \refl_{\chi_V}\right)(z^k);\\ 
\text{FBS:} \; (V = \cH) \quad z^{k+1} &=(1-\lambda_k)z^k +  \lambda_k \prox_{\gamma f} \circ (I_\cH - \gamma \nabla g)(z^k); \\
\text{FBS:} \; (f \equiv 0) \quad z^{k+1} &=(1-\lambda_k)z^k +  \lambda_k P_V \circ (z - \gamma P_V \circ \nabla g\circ P_V)(z^k).
\end{align*}
For general $f, g$ and $V$, the primal DRS and FBS algorithms are not capable splitting Problem~\eqref{eq:FDRSprob} in the same way as~\eqref{eq:TFDRS}. Indeed, the DRS algorithm cannot use the smooth structure of $g$, and the FBS algorithm requires the evaluation of $\prox_{\gamma (f + \chi_V)}(\cdot) = \argmin_{x \in V} \left(f(x) + (1/2\gamma)\|x - \cdot\|^2\right).$ The FDRS algorithm eliminates these difficult problems and replaces them with (possibly) more tractable ones.

\subsection{Proximal, averaged, and FDRS operators}
We briefly review some operator-theoretic properties.

\begin{proposition}\label{prop:basicprox}
Let $\lambda > 0$, let $\gamma > 0$, let $\alpha > 0$, and let $f : \cH \rightarrow (-\infty, \infty]$ be closed, proper, and convex.  
\begin{remunerate}
\item \label{prop:basicprox:part:optprox}{\em Optimality conditions of $\prox$:} Let $x \in \cH$. Then $x^+ = \prox_{\gamma f} (x)$ if, and only if, $\tnabla f(x^+) :=(1/\gamma)(x-x^+) \in \partial f(x^+).$
\item \label{prop:basicprox:part:optproj}{\em Optimality conditions of $\prox_{\chi_V}$:} Let $x \in \cH$. Then $x^+ = \prox_{\gamma \chi_V} (x)$ if, and only if, $\tnabla \chi_V(x^+) :=(1/\gamma)(x-x^+) \in \partial  \chi_V (x^+) .$
Also, $\gamma \tnabla \chi_V(x^+) = P_{V^\perp} x \in V^\perp$.
\item \label{prop:basicprox:part:contract}{\em Averaged operator contraction property:} A map $T : \cH \rightarrow \cH$ is $\alpha$-averaged (see~\eqref{eq:averagednotation}) if, and only if, for all $x, y \in \cH$,
\begin{align}\label{eq:avgdecrease}
\|Tx - Ty\|^2 \leq \|x - y\|^2 - \frac{1-\alpha}{\alpha} \|(I_{\cH}- T)x - (I_{\cH} - T)y\|^2.
\end{align}
\item {\em Composition of averaged operators:}\label{prop:basicprox:part:compositioncontract} Let $\alpha_1, \alpha_2 \in (0, 1)$. Suppose $T_1 : \cH \rightarrow \cH$ and $T_2 : \cH \rightarrow \cH$ are $\alpha_1$ and $\alpha_2$-averaged operators, respectively. Then for all $x, y \in \cH$, the map $T_1\circ T_2 : \cH \rightarrow \cH$ is averaged with parameter
\begin{align}\label{eq:averagedcompositioncoefficient}
\alpha_{1, 2} := \frac{\alpha_1 + \alpha_2 - 2\alpha_1\alpha_2}{1- \alpha_1\alpha_2} \in (0, 1)
\end{align}
\item \label{prop:basicprox:part:wider}{\em Wider relaxations:}  A map $T : \cH \rightarrow \cH$ is $\alpha$-averaged if, and only if, $T_{\lambda}$ (see~\eqref{eq:averagednotation}) is $\lambda \alpha$-averaged for all $\lambda \in (0, 1/\alpha)$.
\item \label{prop:basicprox:part:proxcontraction} {\em Proximal operators are $(1/2)$-averaged:} The operator $\prox_{\gamma f} :  \cH \rightarrow \cH$ is $(1/2)$-averaged and, hence, the operator $\refl_{\gamma f} = 2\prox_{\gamma f} - I_{\cH}$ is nonexpansive.
\item \label{prop:basicprox:part:alphaFDRS} {\em Averaged property of the FDRS operator:} Suppose that $\gamma \in (0, 2\beta)$. Then the operator $\TFDRS$ (see~\eqref{eq:TFDRS}) is $\alphaFDRS := 2\beta/(4\beta - \gamma)$ averaged.
\end{remunerate}
\end{proposition}
\begin{proof}
Parts~\ref{prop:basicprox:part:optprox},~\ref{prop:basicprox:part:optproj},~\ref{prop:basicprox:part:contract},~\ref{prop:basicprox:part:wider}, and~\ref{prop:basicprox:part:proxcontraction} can be found in~\cite{bauschke2011convex}. Part~\ref{prop:basicprox:part:compositioncontract} can be found in~\cite{combettes2014compositions}. Part~\ref{prop:basicprox:part:alphaFDRS} follows from two facts: The operator $(({1}/{2})I_{\cH} + (1/2)\refl_{\gamma f} \circ \refl_{\chi_V})$ is $(1/2)$-averaged  by Part~\ref{prop:basicprox:part:proxcontraction}, and $I - \gamma P_V \circ \nabla g \circ P_V$ is $(\gamma/2\beta)$-averaged by \cite[Proposition 4.1 (ii)]{briceno2012forward}. Thus, Part~\ref{prop:basicprox:part:compositioncontract} proves Part~\ref{prop:basicprox:part:alphaFDRS}.
\end{proof}

\begin{remark}
Later we require $(\lambda_j)_{j \geq 0} \subseteq (0, 1/\alphaFDRS)$ so we hope that $\alphaFDRS$ is small.\cut{Note that the coefficient in~\eqref{eq:averagedcompositioncoefficient} is new~\cite{combettes2014compositions} and improves on the previously known estimate in~\cite[Lemma 2.2]{combettes2004solving}. } Note that the expression for $\alphaFDRS$ is new and improves upon the previous constant: $\max\{{2}/{3}, {2\gamma}/({\gamma + 2\beta})\}.$ See also~\cite[Remark 2.7 (i)]{combettes2014compositions}.
\end{remark}

The proof of the following Proposition is essentially contained in~\cite[Theorem 2.4]{combettes2014compositions}.  We reproduce it in Appendix~\ref{app:prop:compogradientsum} in order to derive a bound. The reader should note the following inequality before reading the proof.
\begin{remark}
Let $\varepsilon \in (0, 1)$. Then it is easy to show that
\begin{align*}
\lambda \leq \frac{(1-\varepsilon)(1+\varepsilon \alpha_{1, 2})}{\alpha_{1, 2}} &\implies \lambda \leq 1/\alpha_{1, 2} - \varepsilon^2 \text{ and } \lambda - 1 \leq \frac{1 - \alpha_{1, 2} \lambda}{\alpha_{1, 2}\varepsilon}. \numberthis\label{eq:lambdaboundwithepsilon}
\end{align*}
\end{remark}
\begin{proposition}\label{prop:compogradientsum}
Let $\alpha_1, \alpha_2 \in (0, 1)$.  Suppose that $T_1 : \cH \rightarrow \cH$ and $T_2 : \cH \rightarrow \cH$ are $\alpha_1$ and $\alpha_2$-averaged operators, respectively, and that $z^\ast$ is a fixed-point of $T_1\circ T_2$. Define $\alpha_{1, 2} \in (0, 1)$ as in~\eqref{eq:averagedcompositioncoefficient}.  Let $z^0 \in \cH$, let $\varepsilon \in (0, 1)$, and consider a sequence $(\lambda_j)_{j \geq 0} \subseteq (0, {(1-\varepsilon)(1+\varepsilon \alpha_{1, 2})}/{\alpha_{1, 2}} )$. Let $(z^j)_{j \geq 0}$ be generated by the following iteration: for all $k \geq 0$, let  $z^{k+1} = (T_1\circ T_2)_{\lambda_k}(z^k).$ Then $$\sum_{i=0}^\infty \lambda_i\|(I_{\cH} - T_2)(z^i) - (I_{\cH} - T_2) (z^\ast)\|^2 \leq \frac{\alpha_{2}(1+1/\varepsilon)\|z^0 - z^\ast\|^2}{1-\alpha_{2}}.$$
\end{proposition}

\subsection{Convergence properties of FDRS}\label{sec:convergence}

The paper \cite{briceno2012forward} assumed the stepsize constraint $\gamma \in (0, 2\beta)$ in order to guarantee convergence of Algorithm~\ref{alg:FDRS}. We now show that the parameter $\gamma$ can (possibly) be increased beyond $2\beta$, which can result in faster practical performance. The proof follows by constructing a new Lipschitz differentiable function $h$ so that the triple $(f, h, V)$ generates the same FDRS operator, $\TFDRS$, as $(f, g, V)$.  This result was not included in~\cite{briceno2012forward}.

\begin{lemma}
Define a function 
\begin{align}\label{eq:function}
h: = g \circ P_V.
\end{align}
Then the FDRS operator associated to $(f, g, V)$ is identical to the FDRS operator associated to $(f, h, V)$. Let $1/\beta_V$ be the Lipschitz constant of $\nabla h$. Then $\beta_V \geq \beta$. In addition, let $\gamma \in (0, 2\beta_V)$. Then $\TFDRS$ is $\alphaFDRS^V$-averaged where
\begin{align}\label{eq:alphaFDRSV}
\alphaFDRSV &:= \frac{2\beta_V}{4\beta_V - \gamma}.
\end{align}
\end{lemma}
{\em Proof.}
The averaged property of $\TFDRS$ and the equivalence of FDRS operators follows from Part~\ref{prop:basicprox:part:alphaFDRS} of Proposition~\ref{prop:basicprox}. The bound $\beta_V \geq \beta$ follows because for all $x, y \in \cH$,
\begin{align*}
\|\nabla h(x) - \nabla h(y)\| &= \|P_V \circ g \circ P_V(x) - P_V \circ g\circ P_V(y)\| \leq  \|\nabla g \circ P_V(x) - \nabla g \circ P_V(y)\| \\
&\leq (1/\beta)\|P_V(x) - P_V(y)\| \leq (1/\beta)\|x-y\| \qquad \endproof
\end{align*}

There are cases where $\beta_V$ is significantly larger than $\beta$. For instance, in the quadratic programming example in~\eqref{eq:cqprogramming}, $\beta$ is the reciprocal of the Lipschitz constant of $Q$, which is the maximal eigenvalue $\lambda_{\mathrm{max}}(Q)$ of $Q$. On the other hand, the gradient $\nabla h = P_V \circ Q \circ P_V$ has rank at most $d- \rank(A)$. Thus, unless the eigenvectors of $Q$ with eigenvalue $\lambda_{\mathrm{max}}(Q)$ lie in the $(d-\rank(A))$-dimensional space $V$, the constant $\beta_V = 1/\lambda_{\mathrm{max}}(P_V \circ Q \circ P_V)$ is larger than $\beta = 1/\lambda_{\mathrm{max}}(Q)$. See Appendix~\ref{app:beta_vcompare} for experimental evidence.

Most of our results do not require that $(z^j)_{j \geq 0}$ converges. However, for completeness we include the following weak convergence result.
\begin{proposition}
Let $\gamma \in (0, 2\beta_V)$, let $(\lambda_j)_{j \geq 0} \subseteq (0, 1/\alphaFDRSV)$, and suppose that $\sum_{i=0}^\infty \lambda_i(1-\lambda_i\alphaFDRSV) = \infty.$ Then $(z^j)_{j \geq 0}$ (from Algorithm~\ref{alg:FDRS}) weakly converges to a fixed-point of $\TFDRS$.
\end{proposition}
\begin{proof}
Apply \cite[Proposition 3.1]{briceno2012forward} with the new averaged parameter $\alphaFDRSV$.
\end{proof}

The following theorem recalls several results on convergence rates for the iteration of averaged operators~\cite{davis2014convergence}.  In addition, we show that $(\lambda_j\|\nabla h(z^j) - \nabla h(z^\ast)\|^2)_{j \geq 0}$ is a summable sequence~\cite{briceno2012forward} whenever $(\lambda_j)_{j \geq 0}$ is chosen properly. 
\begin{theorem}\label{thm:FDRSfacts}
Suppose that $(z^j)_{j \geq 0}$ is generated by Algorithm~\ref{alg:FDRS} with $\gamma \in (0, 2\beta_V)$ and $(\lambda_j)_{j \geq 0} \subseteq (0, 1/\alphaFDRSV)$, and let $z^\ast$ be a fixed-point of $\TFDRS$. Then
\begin{remunerate}
\item {\em Fej\'er monotonicity:} \label{thm:FDRSfacts:part:fejer} the sequence $(\|z^j - z^\ast\|^2)_{j \geq 0}$ is nonincreasing. In addition, for all $z\in \cH$ and $\lambda \in (0,1/\alphaFDRSV)$, we have $\|(\TFDRS)_{\lambda} z - z^\ast\| \leq \|z - z^\ast\|.$
\item {\em Summable fixed-point residual:}  \label{thm:FDRSfacts:part:sumFPR}The sum is finite:
\begin{align*}
\sum_{i=0}^\infty \frac{1-\lambda_i\alphaFDRSV}{\lambda_i\alphaFDRSV}\| z^{i+1} - z^i\|^2 \leq \|z^0 - z^\ast\|^2.
\end{align*}
\item {\em Convergence rates of fixed-point residual:} \label{thm:FDRSfacts:part:convergenceFPR}For all $k \geq 0$, let $\tau_k := (1-\lambda_k\alphaFDRSV)\lambda_k/\alphaFDRSV$.  Suppose that $\underline{\tau} := \inf_{j \geq 0} \tau_j > 0$.  Then for $\lambda > 0$ and $k \geq 0$,
\begin{align}\label{thm:FDRSfacts:part:convergenceFPR:eq}
\|(\TFDRS)_{\lambda}(z^k) - z^k\|^2 &\leq \frac{\lambda^2\|z^0 - z^\ast\|^2}{\underline{\tau} (k+1)} && \mathrm{and} && \|(\TFDRS)_{\lambda}(z^k)- z^k\|^2 = o\left(\frac{1}{k+1}\right).
\end{align}
\item {\em Gradient summability:} \label{thm:FDRSfacts:part:gradientsum} Let $\varepsilon \in (0, 1)$ and suppose that 
\begin{align}\label{eq:lambdainclusion}
(\lambda_j)_{j \geq 0} \subseteq \left(0, \frac{(1-\varepsilon)(1+\varepsilon \alphaFDRSV)}{\alphaFDRSV} \right).
\end{align}
Then the following gradient sum is finite:
\begin{align}\label{eq:gradientsum}
\sum_{i=0}^\infty \lambda_i \| \nabla h(z^i) - \nabla h (z^\ast)\|^2 &\leq \frac{(1+\varepsilon)}{\gamma\varepsilon(2\beta_V-\gamma)}\|z^0 - z^\ast\|^2.
\end{align}
\end{remunerate}
\end{theorem}
\begin{proof}
Parts~\ref{thm:FDRSfacts:part:fejer},~\ref{thm:FDRSfacts:part:sumFPR}, and~\ref{thm:FDRSfacts:part:convergenceFPR} are a direct consequence of \cite[Theorem 1]{davis2014convergence} applied to the $\alphaFDRSV$-averaged operator $\TFDRS$.   Part~\ref{thm:FDRSfacts:part:gradientsum} is a direct consequence of Proposition~\ref{prop:compogradientsum} applied to the $(1/2)$-averaged operator $T_1 := (({1}/{2})I_{\cH} + (1/2)\refl_{\gamma f} \circ \refl_{\chi_V})$ (see Part~\ref{prop:basicprox:part:proxcontraction} of Proposition~\ref{prop:basicprox}) and the $(\gamma/(2\beta_V))$-averaged operator $T_2 := I_\cH - \gamma \nabla h$ (from the Baillon-Haddad Theorem~\cite{baillon1977quelques} and \cite[Proposition 4.33]{bauschke2011convex}).
\end{proof}

We call the following term the fixed-point residual (FPR):
\begin{align}\label{eq:FPR}
\|\TFDRS z^k - z^k\|^2 = \frac{1}{\lambda_k^2}\|z^{k+1} - z^k\|^2
\end{align}

\begin{remark}
Note that the convergence rate proved for $\|\TFDRS z^k - z^k\|^2$ in~\eqref{thm:FDRSfacts:part:convergenceFPR:eq} is sharp for the $\TFDRS$ operator \cite[Section 6.1.1]{davis2014convergence}.
\end{remark}


\section{Subgradients and fundamental inequalities}

In this section, we prove several algebraic identities of the FDRS algorithm.  In addition, we prove a relationship between the FPR and the objective error (Propositions~\ref{prop:FDRSupper} and~\ref{prop:FDRSlower}).

In first-order optimization algorithms, we only have access to (sub)gradients and function values.  Consequently, the FPR is usually the squared norm of a linear combination of (sub)gradients of the objective functions.  For example, the gradient descent algorithm for a smooth function $f$ generates a sequence of iterates by using forward gradient steps: $z^{k+1} := z^k - \nabla f(z^k)$; the FPR is $\|z^{k+1} - z^k\|^2 = \|\nabla f(z^k)\|^2.$

In splitting algorithms, the FPR is more complex because the subgradients are generated via forward-gradient or proximal (backward) steps (see Part~\ref{prop:basicprox:part:optprox} of Proposition~\ref{prop:basicprox}) at different points.  Thus, unlike the gradient descent algorithm where the objective error $f(z^k) - f(x^\ast) \leq \dotp{z^k - x^\ast, \nabla f(x^k)}$ can be bounded with the subgradient inequality, splitting algorithms for two or more functions can only bound the objective error when some or all of the functions are evaluated at separate points --- unless a Lipschitz assumption is imposed. In order to use this Lipschitz assumption, we enforce consensus among the variables, which is why the FPR rate is useful. 

\subsection{A subgradient representation of FDRS}\label{sec:subgradientrep}

\begin{figure}[h!]
  \centering
    \begin{tikzpicture}[scale=2]
  
    \draw[directed, thick] (0, 0) -- (0, 2);
    \draw[directed, thick] (0, 2) -- (2.5, 2);
    \draw[directed, thick] (2.5, 2) -- (2.5, 0);
    \draw[directed, thick] (0, 0) -- (3.5, 0);
    \draw[thick] (2.5, 2) -- (3.75,2);
    \draw[thick] (3.75, 2) -- (3.75, 0);
    \draw[thick] (3.75, 0) -- (2.5, 0);
    \draw[directed, thick] (2.5, 2) -- (3.5, 0);
    
    \draw[fill] (0, 0) circle [radius=.040];
    \draw[fill] (2.5, 0) circle [radius=.040];
    \draw[fill] (0, 2) circle [radius=.040];
    \draw[fill] (2.5, 2) circle [radius=.040];
    \draw[fill] (3.5, 0) circle [radius=.040];    

    \node [below left] at (0, 0) {$z$};
    \node [below] at (2.5, 0) {$\TFDRS (z)$};
    \node [below] at (3.5, 0) {$(\TFDRS)_{\lambda} (z)$};

    \node [below left] at (0, 2) {$x_h$};
    \node [below left] at (2.5,2) {$x_f$};
    \node [left] at (-.05, 1) {$-\gamma \tnabla \chi_{V}(x_h)$};
    \node [left] at (2.45, 1) {$\gamma \tnabla \chi_{V}(x_h)$};
    \node [above] at (1.25, 2.05) {$-\gamma \left(\tnabla \chi_V(x_h) +  \nabla h(x_h) + \tnabla f(x_f)\right)$};
    \node[above] at (1.75, 0) {$\lambda(x_f - x_h)$};
    \end{tikzpicture}
    \caption{A single FDRS iteration, from $z$ to $(\TFDRS )_{\lambda}(z)$ (see Lemma~\ref{lem:FDRSidentities}). Both occurrences of $\tnabla \chi_V(x_h)$ represent the same subgradient $(1/\gamma)P_{V^\perp}z = (1/\gamma)(z - x_h) \in V^\perp$.}
    \label{fig:FDRSsquare}

\end{figure}
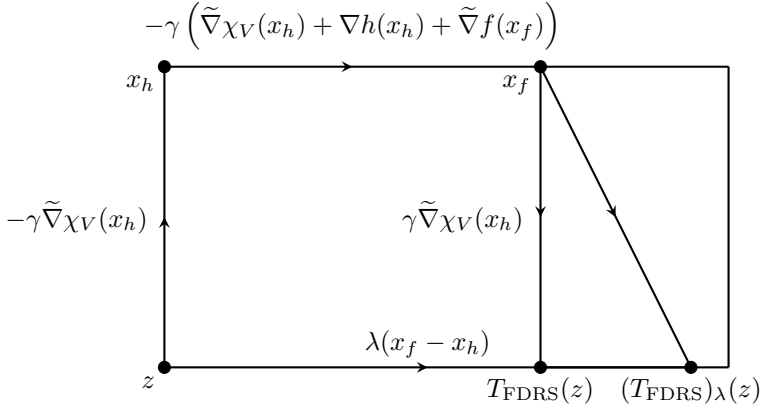

Figure~\ref{fig:FDRSsquare} pictures one iteration of Algorithm~\ref{alg:FDRS}: FDRS projects $z$ onto $V$ to get $x_h = z - \gamma \tnabla \chi_V(x_h)$. The reflection of $z$ across $V$ is $x_h - \gamma \tnabla \chi_{V}(x_h) = z - 2\gamma \tnabla \chi_V(x_h)$. Then FDRS takes a forward-gradient with respect to $\nabla h(x_h)$ from the reflected point $x_h - \gamma \tnabla \chi_V(x_h)$ and a proximal (backward) step with respect to $f$ to get $x_f$.  Finally, we move from $x_f$ to $\TFDRS z$ by traveling along the positive subgradient $\gamma \tnabla \chi_V(x_h)$. 

The following lemma is proved in Appendix~\ref{app:lem:FDRSidentities}.
\begin{lemma}[FDRS identities]\label{lem:FDRSidentities}
Let $z \in \cH$. Define points $x_h$ and $x_f$:
\begin{align}\label{eq:auxpoints}
x_h := P_V z && \mathrm{and} && x_f := \prox_{\gamma f}\circ \refl_{ \chi_V}\circ ( I_\cH - \gamma \nabla h)(z).
\end{align}
Then the identities hold
\begin{align}\label{eq:pointsidentity}
x_h = z - \gamma \tnabla \chi_{V}(x_h) && \mathrm{and} && x_f  = x_h - \gamma \left(\tnabla \chi_V(x_h) +  \nabla h(x_h) +  \tnabla f(x_f)\right)
\end{align}
where $\tnabla \chi_V(x_h) = (1/\gamma)P_{V^\perp}(z)$ and $\tnabla f(x_f)$ is uniquely defined by Part~\ref{prop:basicprox:part:optprox} of Proposition~\ref{prop:basicprox}.
In addition, each FDRS step has the following form:
\begin{align}\label{eq:FPRidentity}
(\TFDRS)_{\lambda}(z) - z = \lambda(x_f - x_h) = - \gamma \lambda \left(\tnabla \chi_V(x_h) +  \nabla h(x_h) +  \tnabla f(x_f)\right).
\end{align}
In particular, $\TFDRS(z) = x_f + \gamma\tnabla \chi_V(x_h)$.
\end{lemma}

\begin{definition}[Ergodic iterates]
Let $(z^j)_{j \geq 0}$ be generated by Algorithm~\ref{alg:FDRS} and define $(x_h^j)_{j \geq 0}$ and $(x_f^j)_{j \geq 0}$ as in~\eqref{eq:auxpoints} (with $z = z^j$). Then define \emph{ergodic iterates:}
\begin{align}\label{eq:ergiterate}
\overline{x}_h^k := \frac{1}{\Lambda_k} \sum_{i=0}^k\lambda_i x_h^i && \text{and}  && \overline{x}_f^k := \frac{1}{\Lambda_k} \sum_{i=0}^k \lambda_i x_f^i 
\end{align}
\vspace{-15pt}
\end{definition}

\subsection{Optimality conditions of FDRS}
The following lemma characterizes the zeros of $\partial f + \nabla h + \partial \chi_V$  in terms of the fixed-points of the FDRS operator. The intuition is the following:  If $z^\ast$ is a fixed-point of $\TFDRS$, then  the base of the rectangle in Figure~\ref{fig:FDRSsquare} has length zero.  Thus, $x^\ast: = x_h^\ast = x_f^\ast$, and if we travel around the perimeter of the rectangle, we will start and begin at $z^\ast$. This argument shows that $ \gamma \tnabla f(x^\ast) + \gamma \nabla h(x^\ast)+  \gamma \tnabla \chi_V(x^\ast)  = 0$, i.e., $x^\ast \in \zer(\partial f +  \nabla h + \partial \chi_V)$.

The following lemma is proved in Appendix~\ref{app:lem:FDRSoptimality}.
\begin{lemma}[FDRS optimality conditions]\label{lem:FDRSoptimality}
The following set equality holds:
\begin{align*}
\zer(\partial f + \nabla h + \partial \chi_V) = \{P_Vz \mid z\in \cH,  \TFDRS z = z\}
\end{align*}
That is, if $z^\ast$ is a fixed-point of $\TFDRS$, then $x^\ast := P_{V} z^\ast = x_h^\ast = x_f^\ast$ is a minimizer of~\eqref{eq:FDRSprob}, and $z^\ast - x^\ast = P_{V^\perp} (z^\ast) = \gamma \tnabla \chi_{V}(x_h^\ast) \in \partial \chi_V(x^\ast).$
\end{lemma}

\subsection{Fundamental inequalities}

In this section, we prove two fundamental inequalities that relate the FPR (see~\eqref{eq:FPR}) to the objective error. \cut{Without any Lipschitz continuity assumption, it seems impossible to bound the true objective error $(f + h + \chi_V)(x) - (f + h + \chi_V)(x^\ast)$.  Thus, we focus on bounding a modified objective error where $h+ \chi_V$ and $f$ are not necessarily evaluated at the same point.  This modified objective error is no longer positive. Therefore, we provide upper and lower bounds in Propositions~\ref{prop:FDRSupper} and~\ref{prop:FDRSlower}.}

Throughout the rest of the paper, we use the following notation: The functions $f$ and $g$ are $\mu_f$ and $\mu_g$-strongly convex, respectively, where we allow $\mu_f$ or $\mu_g$ to be zero (i.e., no strong convexity).  In addition, we assume that $f$ is $(1/\beta_f)$-Lipschitz differentiable, where we allow $\beta_f = 0$.  If $\beta_f > 0$, then $\tnabla f = \nabla f$. With these assumptions, we get the following lower bounds~\cite[Theorem 18.15]{bauschke2011convex}: 
\begin{align*}
&\forall x, y \in \dom(\partial f) &f(x) \geq f(y) + \dotp{x-y, \tnabla f(y)} +S_f(x, y); \numberthis \label{eq:fsbound}\\
&\forall x, y \in \cH  &h(x) \geq h(y) + \dotp{x-y, \nabla h(y)} + S_h(x, y); \numberthis \label{eq:hsbound}
\end{align*}
where $\tnabla f(y) \in \partial f(y)$, and for any $x, y \in \cH$, 
\begin{align*}
S_f(x, y) &:=  \begin{cases} 
\max\left\{\frac{\mu_f}{2} \|x - y\|^2, \frac{\beta_f}{2} \|\nabla f(x) - \nabla f(y)\|^2\right\} & \text{if } \beta_f > 0;\\
\frac{\mu_f}{2} \|x - y\|^2 & \text{otherwise;}
\end{cases} \numberthis \label{eq:Sf} \\
S_h(x, y) &:=  \max\left\{\frac{\mu_g}{2} \|P_Vx - P_Vy\|^2, \frac{\beta_V}{2} \|\nabla h(x) - \nabla h(y)\|^2\right\}. \numberthis \label{eq:Sh}
\end{align*}

See Appendices~\ref{app:prop:FDRSupper},~\ref{app:prop:FDRSlower}, and~\ref{app:cor:strongconvexfundamental} for the proofs of the following inequalities:

\begin{proposition}[Upper fundamental inequality]\label{prop:FDRSupper}
Let $z \in \cH$, let $\lambda > 0$, and let $z^+ := (\TFDRS)_{\lambda} (z)$. Then for all $x \in V \cap \dom(\partial f)$, we have the following inequality:
\begin{align*}
2\gamma\lambda&\left(f(x_f) + h(x_h) - f(x) - h(x) + S_f(x_f, x) + S_h(x_h, x)\right) \\
&\leq \|z - x\|^2 - \|z^+ - x\|^2 + \left(1-\frac{2}{\lambda}\right) \| z^+ - z\|^2 + 2\gamma \dotp{ \nabla h(x_h), z - z^+} \numberthis \label{eq:FDRSupper}
\end{align*}
where $x_f$ and $x_h$ are defined as in Lemma~\ref{lem:FDRSidentities}.
\end{proposition}

\begin{proposition}[Lower fundamental inequality]\label{prop:FDRSlower}
Let $z^\ast \in \cH$ be a fixed-point of $\TFDRS$, and let $x^\ast := P_V z^\ast$\cut{be a minimizer of~\eqref{eq:FDRSprob} (see Lemma~\ref{lem:FDRSoptimality})}. Choose subgradients $\tnabla f(x^\ast) \in \partial f(x^\ast)$ and $\tnabla \chi_V(x^\ast) \in \partial \chi_V(x^\ast)$ with $\tnabla f(x^\ast) + \nabla h(x^\ast) + \tnabla \chi_{V}(x^\ast)  = 0$ (see Lemma~\ref{lem:FDRSoptimality}). Then for all $x_f \in \dom(f)$ and $x_h \in V$, we have 
\begin{align*}
&f(x_f) + h(x_h) - f(x^\ast) - g(x^\ast)  \geq  \dotp{x_f - x_h, \tnabla f(x^\ast)} + S_f(x_f, x^\ast) + S_h(x_h, x^\ast). \numberthis \label{eq:FDRSlower}
\end{align*}
\end{proposition}

\begin{corollary}\label{cor:strongconvexfundamental}
Let $z \in \cH$, let $\lambda > 0$, and let $z^+ := (\TFDRS)_{\lambda} (z)$. Let $z^\ast \in \cH$ be a fixed-point of $\TFDRS$, and let $x^\ast := P_V z^\ast$\cut{ be a minimizer of~\eqref{eq:FDRSprob} (see Lemma~\ref{lem:FDRSoptimality})}. Then with $x_f$ and $x_h$ from Lemma~\ref{lem:FDRSidentities},
\begin{align*}
4\gamma \lambda(S_f(x_f, x^\ast) + S_h(x_h, x^\ast)) &\leq \|z - z^\ast\|^2 - \|z^{+} - z^\ast\|^2 + \left(1-\frac{2}{\lambda}\right) \| z^{+} - z\|^2 \\
&+ 2\gamma \dotp{ \nabla h(x_h) - \nabla h(x^\ast), z - z^{+}}. \numberthis \label{eq:strongconvexupper}
\end{align*}\vspace{-15pt}
\end{corollary}

\section{Objective convergence rates}\label{sec:objectiveconvergencerates}

In this section, we analyze the ergodic and nonergodic convergence rates of the FDRS algorithm applied to~\eqref{eq:FDRSprob}.  

Throughout the rest of the paper, $z^\ast$ will denote an arbitrary fixed-point of $\TFDRS$, and we define a minimizer of~\eqref{eq:FDRSprob} using Lemma~\ref{lem:FDRSoptimality}: $x^\ast := P_V z^\ast.$

All of our bounds will be produced on objective errors of the form:
\begin{align}\label{eq:objectiveerror}
f(x_f^k) + h(x_h^k) - f(x^\ast) - g(x^\ast) && \mathrm{and} && f(x_h^k) + h(x_h^k) - f(x^\ast) - g(x^\ast).
\end{align}
The objective error on the left hand side of~\eqref{eq:objectiveerror} can be negative. Thus, we bound its absolute value. In addition, we bound $\|x_f^k - x_h^k\|$. Because $x_h^k \in V$,  the objective error on the right hand size of~\eqref{eq:objectiveerror} is positive. Consequently, $x_h^k$ is the natural point at which to measure the convergence rate.  To derive such a bound, we assume $f$ is Lipschitz.  Note that in both cases, we have the identity $h(x_h^k) = (g\circ P_V)(x_h^k) = g(x_h^k)$.

\subsection{Ergodic convergence rates}\label{sec:ergodic}

In this section, we analyze the ergodic convergence rate of the FDRS algorithm.  The key idea is to use the telescoping property of the upper and lower fundamental inequalities, together with the summability of the difference of gradients shown in Part~\ref{thm:FDRSfacts:part:gradientsum} of Theorem~\ref{thm:FDRSfacts}.  See Section~\ref{sec:notation} for the distinction between ergodic and nonergodic convergence rates.

\begin{theorem}[Ergodic convergence of FDRS]\label{thm:FDRSergodic}
Let $\gamma \in (0, 2\beta_V)$, let $\varepsilon \in (0, 1)$, and suppose that  $(\lambda_j)_{j \geq 0}$ satisfies~\eqref{eq:lambdainclusion}. Define $(\overline{x}_f^j)_{j \geq 0}$ and $(\overline{x}_h^j)_{j \geq 0}$ as in~\eqref{eq:ergiterate}. Then we have the following convergence rate: for all $k \geq 0$, 
\begin{align*}
\frac{-2\|z^0 - z^\ast\|\|\tnabla f(x^\ast)\|}{\Lambda_k} & \leq f(\overline{x}_f^k) + h(\overline{x}_h^k) - f(x^\ast) - h(x^\ast) \\
&\leq  \frac{\left(\|z^0 - z^\ast\| + 4\gamma\|\nabla h(x^\ast)\| + \frac{(1+\varepsilon)\gamma\|z^0 - z^\ast\|}{\varepsilon^3(2\beta_V - \gamma)}\right) \|z^0 - z^\ast\|}{2\gamma\Lambda_k}. 
\end{align*}
In addition the following feasibility bound holds: $\|\overline{x}_f^k - \overline{x}_h^k\| \leq (2/\Lambda_k)\|z^0 - z^\ast\|.$
\end{theorem}

{\em Proof.}
Fix $k \geq 0$. The feasibility bound follows from Part~\ref{thm:FDRSfacts:part:fejer} of Theorem~\ref{thm:FDRSfacts}:
\begin{align*}
\|\overline{x}_f^k - \overline{x}_h^k\| = \left\|\frac{1}{\Lambda_k} \sum_{i=0}^k \left(z^{i+1} - z^i\right)\right\| = \frac{1}{\Lambda_k}\|z^0 - z^{k+1} \| &\leq \frac{1}{\Lambda_k}\left(\|z^0 - z^\ast\| + \|z^\ast - z^{k+1}\|\right) \\
&\leq \frac{2}{\Lambda_k}\|z^0 - z^\ast\|. \numberthis \label{eq:ergodicfeasibilityboundderivation}
\end{align*}

Now we prove the objective convergence rates. For all $k \geq 0$, let $\eta_k := 2/\lambda_k - 1$.  Note that $\eta_k > 0$ by~\eqref{eq:lambdaboundwithepsilon} because we have $\lambda_k < 1/\alphaFDRSV - \varepsilon^2 \leq 2 - \varepsilon^2$ and $1/\eta_k =  \lambda_k/(2-\lambda_k) \leq \lambda_k/\varepsilon^2$.  Thus, by Cauchy-Schwarz and~\eqref{eq:young}, we have 
\begin{align*}
&2\gamma\dotp{ \nabla h(x_h^k), z^k - z^{k+1}} = 2\gamma \dotp{ \nabla h(x^\ast), z^k - z^{k+1}} +  2\gamma \dotp{\nabla h(x_h^k) - \nabla h(x^\ast), z^k - z^{k+1}} \\
&\leq  2\gamma \dotp{ \nabla h(x^\ast), z^k - z^{k+1}} + \frac{\gamma^2}{\eta_k} \|\nabla h(x_h^k) - \nabla h(x^\ast)\|^2 + \eta_k \|z^k - z^{k+1}\|^2. \numberthis \label{eq:ergodicauxiliarybound}
\end{align*}
Therefore, by Jensen's inequality, the Cauchy-Schwarz inequality,~\eqref{eq:FDRSupper}, and the bound $\|z^{0} - z^{k+1}\|\leq 2\|z^0 - z^\ast\|$ (see~\eqref{eq:ergodicfeasibilityboundderivation}), we have
\begin{align*}
&f(\overline{x}_f^k) + h(\overline{x}_h^k) - f(x^\ast) - h(x^\ast) \leq \frac{1}{\Lambda_k} \sum_{i=0}^k \lambda_i\left (f(x_f^i) + h(x_h^i) - f(x^\ast) - h(x^\ast)\right) \\
&\stackrel{\eqref{eq:FDRSupper}}{\leq} \frac{1}{2\gamma\Lambda_k}\sum_{i=0}^k\left(\|z^i - x^\ast\|^2 - \|z^{i+1} - x^\ast\|^2  -\eta_i \| z^{i+1} - z^i\|^2 + 2\gamma \dotp{ \nabla h(x_h^i), z^i - z^{i+1}}\right) \\
&\stackrel{\eqref{eq:ergodicauxiliarybound}}{\leq} \frac{1}{2\gamma\Lambda_k}\left(\|z^0 - x^\ast\|^2  + 2\gamma \dotp{\nabla h(x^\ast), z^{0} - z^{k+1}} + (\gamma^2/\varepsilon^2) \sum_{i=0}^\infty \lambda_i \|\nabla h(x_h^i) - \nabla h(x^\ast) \|^2\right) \\
&\stackrel{\eqref{eq:gradientsum}}{\leq} \frac{\|z^0 - x^\ast\|^2 + 4\gamma\|\nabla h(x^\ast)\|\|z^0 - z^\ast\| + (1+\varepsilon)\gamma\|z^0 - z^\ast\|^2/(\varepsilon^3(2\beta_V - \gamma))}{2\gamma\Lambda_k}.
\end{align*} 

The lower bound in Proposition~\ref{prop:FDRSlower} and the Cauchy-Schwarz inequality show that
\begin{align*}
f(\overline{x}_f^k) + h(\overline{x}_h^k) - f(x^\ast) - h(x^\ast) \geq \dotp{\overline{x}_f^k - \overline{x}_h^k, \tnabla f(x^\ast)} &\geq -\|\overline{x}_f^k - \overline{x}_h^k\|\|\tnabla f(x^\ast)\| \\
&\geq \frac{-2\|z^0 - z^\ast\|\|\tnabla f(x^\ast)\|}{\Lambda_k}.\qquad \endproof
\end{align*}

In general, $\overline{x}_h^k$ and  $\overline{x}_h^k$ are not in $\dom(f)$. However, the conclusion of Theorem~\ref{thm:FDRSergodic} can be improved if $f$ is Lipschitz continuous. The following proposition gives a sufficient condition for Lipschitz continuity on a ball.

\begin{proposition}[Lipschitz continuity on a ball {\cite[Proposition 8.28]{bauschke2011convex}}]\label{prop:lipschitzCont}
Suppose that $f : \cH \rightarrow (-\infty, \infty]$ is proper and convex.  Let $\rho > 0$, and let $x_0 \in \dom(f)$.  If $\delta = \sup_{x, y \in B(x_0, 2\rho)} |f(x) - f(y)| < \infty$, then $f$ is $({\delta}/{\rho})$-Lipschitz on $B(x_0, \rho)$.
\end{proposition}

To use this fact, we need to show that the sequences $(x_f^j)_{j \geq 0}$, and $(x_h^j)_{j \geq 0}$ are bounded.  Recall that $x_h^s = P_V(z^s)$ and $x_f^s = \prox_{\gamma f}\circ \refl_{\chi_V}\circ ( I_\cH - \gamma \nabla h)(z^s)$ for $s \in \{\ast, k\}$. Proximal, reflection, and forward-gradient maps are nonexpansive (see Proposition~\ref{prop:basicprox}, the Baillon-Haddad Theorem~\cite{baillon1977quelques}, and \cite[Proposition 4.33]{bauschke2011convex}), so we have $\max\{ \|x_f^k - x^\ast\|, \|x_h^k - x^\ast\|\} \leq \|z^k - z^\ast\| \leq \|z^0 - z^\ast\|$ for all $k \geq 0$. Thus, $(x_f^j)_{j \geq 0}, (x_h^j)_{j \geq 0} \subseteq \overline{B(x^\ast, \|z^0 - z^\ast\|)}.$  The ball is convex, so $(\overline{x}_f^j)_{j \geq 0}, (\overline{x}_h^j)_{j \geq 0} \subseteq \overline{B(x^\ast, \|z^0 - z^\ast\|)}.$

\begin{corollary}[Ergodic convergence with Lipschitz $f$]\label{cor:FDRSergodiclipschitz}
Let the notation be as in Theorem~\ref{thm:FDRSergodic}.  Let $L \geq 0$ and suppose $f$ is $L$-Lipschitz on $\overline{B(x^\ast, \|z^0 - z^\ast\|)}$.  Then 
\begin{align*}
0 &\leq f(\overline{x}_h^k) + h(\overline{x}_h^k) - f(x^\ast) - h(x^\ast) \\
&\leq  \frac{\left(\|z^0 - z^\ast\| + 4\gamma\|\nabla h(x^\ast)\| + \frac{(1+\varepsilon)\gamma\|z^0 - z^\ast\|}{\varepsilon^3(2\beta_V - \gamma)}\right) \|z^0 - z^\ast\|}{2\gamma\Lambda_k} + \frac{2L\|z^0 - z^\ast\|}{\Lambda_k}.
\end{align*}
\end{corollary}
{\em Proof.}
The proof follows from by combining the upper bound in Theorem~\ref{thm:FDRSergodic} with the following bound: $
f(\overline{x}_h^k)  \leq f(\overline{x}_f^k) + L\|\overline{x}_f^k - \overline{x}_h^k\| \leq  f(\overline{x}_f^k) + 2L\|z^0 - z^\ast\|/\Lambda_k. \qquad \endproof$

\begin{remark}\label{rem:optimalergodic}
Corollary~\ref{cor:FDRSergodiclipschitz} is sharp~\cite[Proposition 8]{davis2014convergence}.
\end{remark}

\subsection{Nonergodic convergence rates}
In this section, we analyze the nonergodic convergence rate of  FDRS when $(\lambda_j)_{j \geq 0}$ is bounded away from $0$ and $1/\alphaFDRSV$. The proof bounds the inequalities in Propositions~\ref{prop:FDRSupper} and~\ref{prop:FDRSlower} with Theorem~\ref{thm:FDRSfacts}.

\begin{theorem}[Nonergodic convergence of FDRS]\label{thm:FDRSnonergodic}
For all $k \geq 0$, let $\lambda_k \in (0, 1/\alphaFDRSV)$.  Suppose that $\underline{\tau} := \inf_{j \geq 0} (1-\alphaFDRSV\lambda_j)\lambda_j/\alphaFDRSV > 0$. Then
\begin{align*}
 \|x_f^k - x_h^k\| \leq \frac{\|z^0 - z^\ast\|}{\sqrt{\underline{\tau}(k+1)}}, &&&& \|x_f^k - x_h^k\| = o\left(\frac{1}{\sqrt{k+1}}\right),
 \end{align*}
 and 
\begin{align*}
-\frac{\|z^0 - z^\ast\|\|\tnabla f(x^\ast)\|}{\sqrt{\underline{\tau}(k+1)}} &\leq f(x_f^k) + h(x_h^k) - f(x^\ast) - g(x^\ast) \\
&\leq \frac{\left(\|z^{\ast} - x^\ast\| + (1+ {\gamma}/{\beta_V})\|z^0 - z^\ast\| + \gamma\|\nabla h(x^\ast)\|\right)\|z^0 - z^\ast\|}{\gamma\sqrt{\underline{\tau}(k+1)}}, 
\end{align*}
and $|f(x_f^k) + h(x_h^k) - f(x^\ast) - g(x^\ast)| =  o(1/\sqrt{k+1})$. 
\end{theorem}
\begin{proof}
First we note that $\left(\|\nabla h(x_h^j)\|\right)_{j \geq 0}$ is bounded: for all $k \geq 0$,
\begin{align*}
\|\nabla h(x_h^k)\| &\leq \|\nabla h(x_h^k) - \nabla h(x^\ast)\| + \|\nabla h(x^\ast)\| = \|\nabla h(z^k) - \nabla h(z^\ast)\| + \|\nabla h(x^\ast)\| \\
&\leq \frac{1}{\beta_V}\|z^k - z^\ast\| + \|\nabla h(x^\ast)\| \leq \frac{1}{\beta_V}\|z^0 - z^\ast\| + \|\nabla h(x^\ast)\| \numberthis \label{eq:FDRSnonergodichbound}
\end{align*}
because $(\|z^j - z^\ast\|)_{j \geq 0}$ is decreasing (see Part~\ref{thm:FDRSfacts:part:fejer} of Theorem~\ref{thm:FDRSfacts}).

Next fix $k \geq 0$. For any $\lambda > 0$, define $z_\lambda := (\TFDRS)_{\lambda}(z^k)$.\cut{ It follows that, $z_\lambda - z^k = \lambda( \TFDRS(z^k) - z^k)$ and, hence, we can bound the size $z_\lambda - z^k$ with~\eqref{thm:FDRSfacts:part:convergenceFPR:eq}.} Observe that $x_f^k$ and $x_h^k$ do not depend on the value of $\lambda_k$. Therefore, by Proposition~\ref{prop:FDRSupper} and  Lemma~\ref{lem:FDRSidentities},
\begin{align*}
& f(x_f^k) + h(x_h^k) - f(x^\ast) - g(x^\ast)  \\
&\leq \inf_{\lambda \in [0, 1/\alphaFDRSV)}\frac{1}{2\gamma\lambda}\biggl(\|z^k - x^\ast\|^2 - \|z_\lambda - x^\ast\|^2 + \left(1-\frac{2}{\lambda}\right) \| z_\lambda - z^k\|^2  \\
&\quad\quad\quad\quad\quad\quad\quad\quad\quad + 2\gamma \dotp{\nabla h(x_h^k), z^k - z_\lambda}\biggr) \\
&\stackrel{\eqref{eq:cosinerule}}{=}  \inf_{\lambda \in [0, 1/\alphaFDRSV)} \frac{1}{2\gamma\lambda}\biggl(2\dotp{z_\lambda - x^\ast, z^k - z_\lambda} + 2\left(1 - \frac{1}{\lambda}\right)\| z_\lambda - z^k\|^2\\
&\quad\quad\quad\quad\quad\quad\quad\quad\quad + 2\gamma \dotp{\nabla h(x_h^k), z^k - z_\lambda}\biggr) \\
&\stackrel{\eqref{eq:FDRSnonergodichbound}}{\leq}  \frac{1}{2\gamma}\left(2\dotp{z_1 - x^\ast, z^k - z_1} + 2\gamma \left(\frac{1}{\beta_V}\|z^0 - z^\ast\| + \|\nabla h(x^\ast)\|\right) \|z_1 - z^k\|\right) \numberthis \label{eq:littleoupper}\\
&\stackrel{\eqref{thm:FDRSfacts:part:convergenceFPR:eq}}{\leq} \frac{\left(\|z_1 - x^\ast\| + ({\gamma}/{\beta_V})\|z^0 - z^\ast\| + \gamma\|\nabla h(x^\ast)\|\right)\|z^0 - z^\ast\|}{\gamma\sqrt{\underline{\tau}(k+1)}}  \\
&\leq  \frac{\left(\|z^{\ast} - x^\ast\| + (1+ {\gamma}/{\beta_V})\|z^0 - z^\ast\| + \gamma\|\nabla h(x^\ast)\|\right)\|z^0 - z^\ast\|}{\gamma\sqrt{\underline{\tau}(k+1)}}
\end{align*}
where we use $\|z_1 - x^\ast\| \leq \|z_1- z^\ast\| + \|z^\ast - x^\ast\| \leq \|z^0 - z^\ast\| + \|z^\ast - x^\ast\|$ (Theorem~\ref{thm:FDRSfacts}).

The lower bound follows from~\eqref{eq:FDRSlower} and Part~\ref{thm:FDRSfacts:part:convergenceFPR} of Theorem~\ref{thm:FDRSfacts}:
\begin{align*}
f(x_f^k) + h(x_h^k) - f(x^\ast) - g(x^\ast) \geq \dotp{ x_f^k - x_h^k, \tnabla f(x^\ast)} &= \frac{1}{\lambda_k}\dotp{z^{k+1} - z^k, \tnabla f(x^\ast)} \numberthis \label{eq:littleolower} \\
&\stackrel{\eqref{thm:FDRSfacts:part:convergenceFPR:eq}}{\geq} -\frac{\|z^0 - z^\ast\|\|\tnabla f(x^\ast)\|}{\sqrt{\underline{\tau}(k+1)}}.
\end{align*}

The $o(1/\sqrt{k+1})$ rates follow from~\eqref{eq:littleoupper} and~\eqref{eq:littleolower}, and the corresponding rates for the FPR in~\eqref{thm:FDRSfacts:part:convergenceFPR:eq}. The bounds on $x_f^k - x_h^k$ follow from $x_f^k - x_h^k = \TFDRS z^k - z^k$.
\end{proof}

If $f$ is Lipschitz continuous, we can evaluate the entire objective function at $x_h^k$.  The proof of the following corollary is analogous to  Corollary~\ref{cor:FDRSergodiclipschitz}. We ask the reader to recall from Section~\ref{sec:ergodic} that $(x_f^j)_{j \geq 0}, (x_h^j)_{j \geq 0} \subseteq \overline{B(x^\ast, \|z^0 - z^\ast\|)}.$

\begin{corollary}[Nonergodic convergence with Lipschitz $f$]\label{cor:FDRSnonergodicLipschitz}
Let the notation be as in Theorem~\ref{thm:FDRSnonergodic}.  Let $L \geq 0$ and suppose $f$ is $L$-Lipschitz on $\overline{B(x^\ast, \|z^0 - z^\ast \|)}$.  Then 
\begin{align*}
0&\leq f(x_h^k) + h(x_h^k) - f(x^\ast) - h(x^\ast) \\
&\leq  \frac{\left(\|z^{\ast} - x^\ast\| + (1+ {\gamma}/{\beta_V})\|z^0 - z^\ast\| + \gamma\|\nabla h(x^\ast)\| \right)\|z^0 - z^\ast\|}{\gamma\sqrt{\underline{\tau}(k+1)}} + \frac{L\|z^0 - z^\ast\|}{\sqrt{\underline{\tau}(k+1)}},
\end{align*}
and $f(x_h^k) + h(x_h^k) - f(x^\ast) - h(x^\ast) = o(1/\sqrt{k+1})$.
\end{corollary}
\begin{proof}
Combine the upper bound in Theorem~\ref{thm:FDRSnonergodic} with the following bound: $f(x_h^k)  \leq f(x_f^k) + L\|x_f^k - x_h^k\| \leq  f(x_f^k)+ L\|z^0 - z^\ast\|/\sqrt{\underline{\tau}(k+1)}.$ The $o(1/\sqrt{k+1})$ rate follows because $\|x_f^k - x_h^k\| = \|\TFDRS z^k - z^k\| = o(1/\sqrt{k+1})$ (see~\eqref{eq:FPRidentity} and~\eqref{thm:FDRSfacts:part:convergenceFPR:eq}) and $|f(x_f^k) +  h(x_h^k) - f(x^\ast) - h(x^\ast) | = o(1/\sqrt{k+1})$ (see Theorem~\ref{thm:FDRSnonergodic}).
\end{proof}

\begin{remark}\label{rem:optimalnonergodic}
Corollary~\ref{cor:FDRSnonergodicLipschitz} is sharp \cite[Theorem 11]{davis2014convergence}.
\end{remark}

\section{Strong convexity}

In this section, we show that $(x_f^j)_{j \geq 0}$, $(x_h^j)_{j \geq 0}$, and their ergodic variants converge strongly whenever $f$ or $g$ is strongly convex. The techniques in this section are similar to those in Section~\ref{sec:objectiveconvergencerates}, so we defer the proof to Appedix~\ref{app:thm:strongconvexity}

\begin{theorem}[Auxiliary term bound]\label{thm:strongconvexity}
Let $\gamma \in (0, 2\beta_V)$, let $(\lambda_j)_{j \geq 0} \subseteq (0, 1/\alphaFDRSV),$ let $z^0 \in \cH$, and suppose that $(z^j)_{j \geq 0}$ is generated by Algorithm~\ref{alg:FDRS}. Then
\begin{remunerate}
\item {\em ``Best" iterate convergence:} Let $\varepsilon \in (0, 1)$ and suppose that  $(\lambda_j)_{j \geq 0}$ satisfies~\eqref{eq:lambdainclusion}.  If $\underline{\lambda} := \inf_{j \geq 0} \lambda_j > 0$, then
\begin{align*}
& \min_{0 \leq j \leq k}\left(S_f(x_f^{j}, x^\ast) + S_h(x_h^{j}, x^\ast)\right) \leq \frac{\left(1 + \frac{(1+\varepsilon)\gamma}{\varepsilon^3(2\beta_V-\gamma)} \right)\|z^0 - z^\ast\|^2}{4\gamma\underline{\lambda}(k+1)}.
\end{align*}
and $\min_{0 \leq j \leq k}S_f(x_f^{j}, x^\ast) = o({1}/{(k+1)})$ and $\min_{0 \leq j \leq k}S_h(x_h^{j}, x^\ast) = o({1}/{(k+1)})$.
\item {\em Ergodic convergence:}
If $\varepsilon \in (0, 1)$, and $(\lambda_j)_{j \geq 0}$ satisfies~\eqref{eq:lambdainclusion}, then
\begin{align*}
\frac{\mu_f}{2}\|\overline{x}_f^k - x^\ast\|^2 + \frac{\mu_h}{2} \|\overline{x}_h^k - x^\ast\|^2 &\leq \frac{\left(1 + \frac{(1+\varepsilon)\gamma}{\varepsilon^3(2\beta_V-\gamma)} \right)\|z^0 - z^\ast\|^2}{4\gamma\Lambda_k}.
\end{align*}
\item {\em Nonergodic convergence:}\label{thm:strongconvexity:part:nonergodic}
If $\underline{\tau} := \inf_{j \geq 0} (1-\alphaFDRSV\lambda_j)\lambda_j/\alphaFDRSV > 0$, then $S_f(x_f^k, x^\ast) + S_h(x_h^k, x^\ast) = o(1/\sqrt{k+1})$ and 
\begin{align*}
S_f(x_f^k, x^\ast) + S_h(x_h^k, x^\ast) &\leq\frac{(1+ {\gamma}/{\beta_V})\|z^0 - z^\ast\|^2}{2\gamma\sqrt{\underline{\tau}(k+1)}},
\end{align*}
\end{remunerate}
\end{theorem}

\begin{remark}
See Section~\ref{sec:slowrates} for a proof that the nonergodic ``best" rates are sharp.  It is not clear if we can improve the general nonergodic rates to $o(1/(k+1))$.
\end{remark}


\section{Lipschitz differentiability}

In this section, we assume $f$ is smooth:

\begin{assump}\label{assump:fdifferentiable}
 $f$ is differentiable and $\nabla f$ is $(1/\beta_f)$-Lipschitz where $\beta_f > 0$.
\end{assump}

Under Assumption~\ref{assump:fdifferentiable}, we will show that the objective value 
\begin{align*}
f(x_h^k) + h(x_h^k) - f(x^\ast) - h(x^\ast) = f(x_h^k) + g(x_h^k) - f(x^\ast) - g(x^\ast)
\end{align*}
is summable. Therefore, by \cite[Lemma 3]{davis2014convergence} the minimal objective error after $k$ iterations is of order $o(1/(k+1))$. We will need the following upper bound to prove this. See Appendix~\ref{app:eq:newfundamentalinequality} for the proof. 
\begin{proposition}[Fundamental inequality under Assumption~\ref{assump:fdifferentiable}]\label{eq:newfundamentalinequality}
If $\gamma \in (0, 2\beta_V)$, $\lambda > 0$,  $z \in \cH$,  $z^+ := (\TFDRS)_{\lambda}(z)$,  $z^\ast$ is a fixed-point of $\TFDRS$, and $x^\ast = P_Vz^\ast$, then
\begin{align*}
&2\gamma\lambda(f(x_h) + h(x_h) - f(x^\ast) - g(x^\ast)) \\
&\leq \begin{cases}
\|z - z^\ast\|^2 - \|z^+ - z^\ast\|^2 +   \left(1 + \frac{\gamma - \beta_f}{\beta_f\lambda}\right)\|z-z^+\|^2 \\
+ 2\gamma \dotp{ \nabla h(x_h) - \nabla h(x^\ast), z - z^{+}} & \text{if } \gamma \leq \beta_f \\
\left(1 + \frac{\gamma - \beta_f}{2\beta_f}\right) (\|z - z^\ast\|^2 - \|z^+ - z^\ast\|^2 +   \|z-z^+\|^2) &  \\
+ 2\gamma\left(1 + \frac{\gamma - \beta_f}{2\beta_f}\right)\dotp{ \nabla h(x_h) - \nabla h(x^\ast), z - z^{+}} & \text{if } \gamma > \beta_f.
\end{cases} \numberthis \label{eq:lipshitzfundamentalinequality}
\end{align*}
\end{proposition}

The next theorem shows that the upper bound in Proposition~\ref{eq:newfundamentalinequality} is summable and, as a consequence, we will have $o(1/(k+1))$ convergence.

\begin{theorem}[Convergence rates under Assumption~\ref{assump:fdifferentiable}]\label{thm:lipschitzbest}
Let $\gamma \in (0, 2\beta_V)$, let $\varepsilon \in (0, 1)$, and suppose  $(\lambda_j)_{j \geq 0}$ satisfies~\eqref{eq:lambdainclusion}. Suppose that $\underline{\tau} := \inf_{j \geq 0}\{ (1-\alphaFDRSV\lambda_j)\lambda_j/\alphaFDRSV\} > 0$ and let $\underline{\lambda} := \inf_{j \geq 0} \lambda_j  > 0$. Let $z^0 \in \cH$, let $z^\ast$ be a fixed-point of $\TFDRS$, and let $x^\ast := P_Vz^\ast$. Then
$$\min_{0 \leq j \leq k} \left(f(x_h^{j}) + h(x_h^{j}) - f(x^\ast) - h(x^\ast)\right) = o\left(\frac{1}{k+1}\right).$$
 \end{theorem}
\begin{proof}
Let $\delta := \inf_{j \geq 0} \left\{(1-\lambda_j\alphaFDRSV)/(\lambda_j\alphaFDRSV)\right\}$. Note that $0 < \delta < \infty$ because $\underline{\tau} > 0$. Now, recall that, by Part~\ref{thm:FDRSfacts:part:sumFPR} of Theorem~\ref{thm:FDRSfacts}, we have
\begin{align*}
 \sum_{i=0}^\infty \| z^{i+1} - z^i\|^2 \leq  \frac{1}{\delta} \sum_{i=0}^\infty \frac{1-\lambda_i\alphaFDRSV}{\lambda_i\alphaFDRSV}\| z^{i+1} - z^i\|^2 \leq \frac{1}{\delta} \|z^0 - z^\ast\|^2.
\end{align*}
Next, we use the Cauchy-Schwarz inequality and~\eqref{eq:young} to show that
\begin{align*}
\sum_{i=0}^\infty  2\gamma \dotp{ \nabla h(x_h^i) - \nabla h(x^\ast), z^i - z^{i+1}}  &\leq \sum_{i=0}^\infty \left(\lambda_i \gamma^2\|\nabla h(x_h^i) - \nabla h(x^\ast)\|^2 + \frac{1}{\lambda_i} \|z^i- z^{i+1}\|^2\right) \\
&\stackrel{\eqref{eq:gradientsum}}{\leq} \left( \frac{(1+\varepsilon)\gamma}{\varepsilon(2\beta_V-\gamma)} + \frac{1}{\underline{\lambda}\delta}  \right)\|z^0 - z^\ast\|^2.
\end{align*}
If we combine the previous two sum bounds with~\eqref{eq:lipshitzfundamentalinequality}, we get
\begin{align*}
&\sum_{i=0}^\infty (f(x_h^i) + h(x_h^i) - f(x^\ast) - h(x^\ast)) \\
&\leq \frac{\left(1 + \frac{1}{\delta} + \frac{(1+\varepsilon)\gamma}{\varepsilon(2\beta_V-\gamma)} + \frac{1}{\underline{\lambda}\delta}\right)\|z^0 - z^\ast\|^2}{2\gamma\underline{\lambda}} \times \begin{cases} 
1& \text{if } \gamma \leq \beta_f; \\
\left(1 + \frac{\gamma - \beta_f}{2\beta_f}\right) & \text{if } \gamma > \beta_f.
\end{cases}
\end{align*}
The convergence rate now follows from \cite[Lemma 3]{davis2014convergence}.
\end{proof}

\begin{remark}
Theorem~\ref{thm:lipschitzbest} is sharp under Assumption~\ref{assump:fdifferentiable}~\cite[Theorem 12]{davis2014convergence}.
\end{remark}

\section{Linear convergence}

In this section, we prove FDRS converges linearly when $\beta_f(\mu_g + \mu_f) > 0$. 

\begin{theorem}[Linear convergence]\label{thm:linearconvergence}
Let $\gamma \in (0, 2\beta_V)$, let $(\lambda_j)_{j \geq 0} \subseteq (0, 1/\alphaFDRSV)$, let $z^0 \in \cH$, let $z^\ast$ be a fixed-point of $\TFDRS$, and let $x^\ast := P_Vz^\ast$. Let $c > 1/2$, let $\gamma < \beta_V/c$, and let $(\lambda_j)_{j \geq 0} \subseteq (0,  (2c - 1)/c)$.  For all $\lambda \in (0, (2c - 1)/c)$, define
\begin{align*}
C_1(\lambda) &:= \left(1 - \frac{\lambda}{3}\min\left\{\frac{\gamma\mu_g}{(1+\gamma/\beta_V)^2}, \frac{\beta_f}{\gamma}, \frac{2c-1}{c} - \lambda \right\}\right)^{1/2}; \\
C_2(\lambda) &:= \left(1 - \frac{\lambda}{3}\min\left\{\frac{\gamma \mu_f}{(1+\gamma/\beta_f)^2}, \frac{\beta_V - c\gamma}{\gamma}, \frac{1}{4}\left(\frac{2c-1}{c} - \lambda\right) \right\}\right)^{1/2}.
\end{align*}

Then for all $k \geq 0$, we have
\begin{align*}
\|z^{k+1} - z^\ast\| &\leq  \|z^{k} - z^\ast\|\times \begin{cases}
C_1(\lambda_k) & \text{if } \mu_g\beta_f > 0;\\
C_2(\lambda_k) & \text{if } \mu_f\beta_f > 0;
\end{cases}\numberthis \label{eq:contraction}\\
\|z^{k+1} - z^\ast\| &\leq \|z^0 - z^\ast\| \times\begin{cases}
\prod_{i=0}^kC_1(\lambda_i) & \text{if } \mu_g\beta_f > 0;\\
\prod_{i=0}^k C_2(\lambda_i) & \text{if } \mu_f\beta_f > 0.
\end{cases}
\end{align*}
\end{theorem}
\begin{proof}
~\eqref{eq:strongconvexupper} shows that for all $k \geq 0$, we have
\begin{align*}
&\gamma \lambda_k\mu_f\|x_f^k - x^\ast\|^2 + \gamma \lambda_k \beta_f \|\nabla f(x_f^k) - \nabla f(x^\ast)\|^2 \\
&+ \gamma\lambda_k\mu_g \|x_h^k - x^\ast\|^2 + \gamma\lambda_k\beta_V\|\nabla h(x_h^k) - \nabla h(x^\ast)\|^2  \\
&\leq \|z^k - z^\ast\|^2 - \|z^{k+1} - z^\ast\|^2 + \left(1-\frac{2}{\lambda_k}\right) \| z^{k+1} - z^k\|^2 \\
&+ 2\gamma \dotp{ \nabla h(x_h^k) - \nabla h(x^\ast), z^k - z^{k+1}}.
\end{align*}
In addition, by the Cauchy-Schwarz inequality and~\eqref{eq:young}, we have
\begin{align*}
2\gamma \dotp{ \nabla h(x_h^k) - \nabla h(x^\ast), z^k - z^{k+1}} &\leq c\gamma^2 \lambda_k\|\nabla h(x_h^k) - \nabla h(x^\ast)\|^2 + \frac{1}{c\lambda_k} \|z^k - z^{k+1}\|^2.
\end{align*}
Therefore, for all $k \geq 0$, 
\begin{align*}
&\gamma \lambda_k\mu_f\|x_f^k - x^\ast\|^2 + \gamma \lambda_k \beta_f \|\nabla f(x_f^k) - \nabla f(x^\ast)\|^2 \\
&+ \gamma\lambda_k\mu_g \|x_h^k - x^\ast\|^2 + \gamma\lambda_k(\beta_V - c\gamma)\|\nabla h(x_h^k) - \nabla h(x^\ast)\|^2  \\
&\leq \|z^k - z^\ast\|^2 - \|z^{k+1} - z^\ast\|^2 + \left(1-\frac{2c - 1}{c\lambda_k}\right) \| z^{k+1} - z^k\|^2.
\end{align*}
Recall that we assume $1-(2c - 1)/(c\lambda_k) < 0$ and $\beta_V - c\gamma > 0$.

Now suppose that $\beta_f\mu_g > 0$.  The following identity follows from from Lemma~\ref{lem:FDRSidentities}:
\begin{align*}
z^k = \TFDRS(z^k)+ (z^k - \TFDRS(z^k)) = x_h^k - \gamma \nabla h(x_h^k) - \gamma \nabla f(x_f^k) + \frac{1}{\lambda_k}(z^k - z^{k+1}).
\end{align*}
This identity results from tracing the perimeter of Figure~\ref{fig:FDRSsquare} from $x_h$ to $x_f$ to $\TFDRS z^k$ to $z^k$. Likewise, we have $z^\ast = x^\ast - \gamma \nabla h(x^\ast) - \gamma \nabla f(x^\ast)$. 

Note that
\begin{align*}
\|(x_h^k - \gamma \nabla h(x_h^k)) - (x^\ast - \gamma \nabla h(x^\ast))\| &\leq \|x_h^k - x^\ast\| + \gamma\|\nabla h(x_h^k) - \nabla h(x^\ast)\| \\
&\leq (1 + \gamma/\beta_V)\|x_h^k - x^\ast\|.\numberthis\label{eq:Iminusnablahcontraction}
\end{align*}
Now, fix $k \geq 0$, and let  $C_1' := 3\max\left\{ (1+\gamma/\beta_V)^2/(\gamma\lambda_k\mu_g), \gamma^2/(\gamma \lambda_k\beta_f), (1/\lambda_k^2)\left(\frac{2c - 1}{c\lambda_k}-1\right)^{-1}\right\}.$ By the convexity of $\|\cdot\|^2$, we have
\begin{align*}
&\|z^{k} - z^\ast\|^2 \leq 3(1+\gamma/\beta_V)^2\|x_h^k - x^\ast\|^2 + 3\gamma^2\|\nabla f(x_f^k) - \nabla f(x^\ast)\|^2 + \frac{3}{\lambda_k^2}\|z^{k+1} - z^k\|^2 \\
&\leq C_1' \biggl(\gamma\lambda_k\mu_g \|x_h^k - x^\ast\|^2 + \gamma \lambda_k \beta_f \|\nabla f(x_f^k) - \nabla f(x^\ast)\|^2 +\left(\frac{2c - 1}{c\lambda_k} -1\right)\|z^{k+1} - z^k\|^2\biggr) \\
&\leq C_1'   \|z^k - z^\ast\|^2 - C_1'\|z^{k+1} - z^\ast\|^2.
\end{align*}
 Therefore, $\|z^{k+1} - z^\ast\| \leq \left(1 - (1/C_1')\right)^{1/2}\|z^{k} - z^\ast\|.$

Now assume that $\beta_f\mu_f > 0$. Observe that:
\begin{align*}
z^{k} &= x_h^k - \gamma \nabla h(x_h^k) - \gamma \nabla f(x_f^k) + \frac{1}{\lambda_k}(z^k - z^{k+1}) \\
&= x_f^k -  \gamma \nabla h(x_h^k) - \gamma \nabla f(x_f^k) + \frac{2}{\lambda_k}(z^k - z^{k+1})
\end{align*}
where we use the identity $x_h^k - x_f^k = (1/\lambda_k)(z^k - z^{k+1})$ (see~\eqref{eq:FPRidentity}). The proof of this case is similar to the case $\beta_f\mu_h > 0$ except that we use the above identity for $z^k$, the bound $\|(x_f^k - \gamma \nabla f(x_f^k)) - (x^\ast - \gamma \nabla f(x^\ast))\|^2\leq (1+\gamma/\beta_f)^2\|x_f^k - x^\ast\|^2$, and the constant $C_2' := 3\max\left\{ (1+\gamma/\beta_f)^2/(\gamma\lambda_k\mu_f), \gamma^2/(\gamma\lambda_k(\beta_V - c\gamma)), (4/\lambda_k^2)\left(\frac{2c - 1}{c\lambda_k} - 1\right)^{-1}\right\}$
in place of $C_1'$. Then the contraction $\|z^{k+1} - z^\ast\| \leq \left(1- 1/C_2'\right)^{1/2}\|z^k - z^\ast\|$ follows.

In both cases, the linear rate for $(z^j)_{j \geq 0}$ follows by unfolding~\eqref{eq:contraction}.
\end{proof}

\begin{remark}
Note that smaller $c$ lead to larger $\gamma$ and smaller $(\lambda_j)_{j \geq 0}$, while larger $c$ lead to smaller $\gamma$ and larger $(\lambda_j)_{j \geq 0}$. 
\end{remark}

\subsection{Arbitrarily slow convergence for strongly convex problems}\label{sec:slowrates}
In general, we cannot expect linear convergence of FDRS when $f$ is not differentiable---even if $f$ and $g$ are strongly convex. In this section, we construct an example to prove this claim.\cut{We also show that FDRS applied to this example converges \emph{arbitrarily slowly.}  } The following example is based on~\cite[Section 7]{bauschke2013rate} and~\cite[Example 1]{davis2014convergence}.

\subsubsection*{A family of slow examples}
Let $\cH := \ell_2^2(\vN) = \vR^2 \oplus \vR^2\oplus \cdots$. Let $R_{\theta}$ denote counterclockwise rotation in $\vR^2$ by $\theta$ degrees.  Let $e_0 := (1, 0)$ denote the standard unit vector, and let $e_{\theta} := R_\theta e_0$.  Let $(\theta_j)_{j\geq0}$be a sequence of angles in $(0, {\pi}/{2}]$ such that $\theta_i \rightarrow 0$ as $i \rightarrow \infty$.  For all $i \geq 0$, let $c_i := \cos(\theta_i)$. We let 
\begin{align}
V := \vR e_0 \oplus \vR e_0 \oplus \cdots  && \mathrm{and} && U := \vR e_{\theta_0} \oplus \vR e_{\theta_1} \oplus \cdots.
\end{align}
Note that \cite[Section 7]{bauschke2013rate} proves the projection identities
\begin{align*}
(P_U)_i &= \begin{bmatrix} \cos^2(\theta_i) & \sin(\theta_i)\cos(\theta_i) \\ \sin(\theta_i)\cos(\theta_i) & \sin^2(\theta_i) \end{bmatrix} && \mathrm{and} && (P_V)_i = \begin{bmatrix} 1 & 0 \\ 0 & 0\end{bmatrix},
\end{align*}

We now begin our extension of this example. Choose $a \geq 0$ and set $f := \chi_U + ({a}/{2})\|\cdot\|^2$ and $g := ({1}/{2})\|\cdot\|^2.$ Note that $\mu_g = 1$ and $\mu_f = a$.  In addition, for $h := g\circ P_V$, we have $(\nabla h(x))_i = (P_V \circ I_{\cH} \circ P_V)_i = (P_V)_i.$ Thus, $\nabla h$ is $1$-Lipschitz, and, hence, $\beta_V = 1$ and we can choose $\gamma = 1 < 2\beta_V$. Therefore, $\alphaFDRSV = 2\beta_V/(4\beta_V - \gamma) = 2/3$, so we can choose $\lambda_k \equiv 1 < 1/\alphaFDRSV$. We also note that $\prox_{\gamma f} = (1/(1+a)) P_U$. 

Define $N : \cH \rightarrow \cH$ on each 2-dimensional component of $\cH$ as follows: for all $i \geq 0$,
\begin{align*}
&(N)_i := \left(\frac{1}{2}I_{\cH} + \frac{1}{2}\refl_{\gamma f} \circ \refl_{\chi_V}\right)_i = \frac{1}{a+1}(P_U)_i(2(P_V)_i - I_{\vR^2}) + I_{\vR^2} - (P_V)_i \\
&= \frac{1}{a+1}(P_U)_i\begin{bmatrix} 1 & 0 \\ 0 & -1\end{bmatrix} + \begin{bmatrix} 0 & 0 \\ 0 & 1\end{bmatrix} = \frac{1}{a+1}\begin{bmatrix} \cos^2(\theta_i) & -\sin(\theta_i)\cos(\theta_i) \\ \sin(\theta_i)\cos(\theta_i) &  \cos^2(\theta_i) + a\end{bmatrix}
\end{align*}
where the second equality follows by direct expansion. Therefore, we have
\begin{align}\label{eq:TFDRSsubspace}
\TFDRS = N \circ (I - P_V) = \bigoplus_{i \geq 0} \; \frac{1}{a+1}\begin{bmatrix} 0 & -\sin(\theta_i)\cos(\theta_i) \\ 0 &\cos^2(\theta_i) + a \end{bmatrix}.
\end{align}
Note that for all $i \geq 0$, the operator $(\TFDRS)_i$ has eigenvector 
\begin{align*}
z_i := \left(-\frac{\cos(\theta_i)\sin(\theta_i)}{a + \cos^2(\theta_i)}, 1\right) \numberthis\label{eq:eigenvalueFDRS}
\end{align*}
with eigenvalue $b_i := (a + c_i^2)/(a+1) < 1$. Each component also has the eigenvector $(1, 0)$ with eigenvalue $0$. Thus, the only fixed-point of $\TFDRS$ is $0 \in \cH$.  Finally,
\begin{align}\label{eq:eigenvectornorm}
\|z_i\|^2 = \frac{c_i^2(1-c_i^2)}{(a + c_i^2)^2} + 1 && \mathrm{and} && \|(P_V)_iz_i\|^2 =  \frac{c_i^2(1-c_i^2)}{(a + c_i^2)^2}.
\end{align}

\subsubsection*{Slow convergence proofs}
We know that $z^{k+1} - z^k \rightarrow 0$ from~\eqref{thm:FDRSfacts:part:convergenceFPR:eq}. Therefore, because $\TFDRS$ is linear, \cite[Proposition 5.27]{bauschke2011convex} proves the following lemma.
\begin{lemma}[Strong convergence for linear operators]
Any sequence $(z^j)_{j \geq 0}\subseteq \cH$ generated by the $\TFDRS$ operator in~\eqref{eq:TFDRSsubspace} converges strongly to $0$. Consequently, the sequences $(x_h^j)_{j \geq 0} = (P_Vz^j)_{j \geq 0}$ and $(x_f^j)_{j \geq 0}$ converge strongly to zero.
\end{lemma}

\begin{lemma}[Slow sequences {\cite[Lemma 6]{davis2014convergence}}]\label{lem:slowconvergencesequence}
Suppose that $F : \vR_+ \rightarrow (0, 1)$ is a function that is strictly decreasing to zero such that $\{1/(j+1) \mid j \in \vN \backslash \{0\}\} \subseteq \range(F)$ Then there exists a monotonic sequence $(b_j)_{j \geq 0} \subseteq (0, 1)$ such that $b_k \rightarrow 1^-$ as $k \rightarrow \infty$  and an increasing sequence $(n_j)_{j \geq 0} \subseteq \vN \cup \{0\}$ such that for all $k \geq 0$, $$\frac{b_{n_k}^{k+1}}{(n_k+1)} > e^{-1}F(k+1).$$
\end{lemma}
The following is a simple corollary of Lemma~\ref{lem:slowconvergencesequence}.
\begin{corollary}\label{cor:slowconvergencesequence}
Let the notation be as in Lemma~\ref{lem:slowconvergencesequence}. Then for all $\eta \in (0, 1)$, we can find a sequence $(b_j)_{j \geq 0} \subseteq (\eta, 1)$ that satisfies the conditions of the lemma. 
\end{corollary}
\begin{proof}
For any $\varepsilon \in (0, 1 - \eta)$, replace the sequence $(b_j)_{j \geq 0}$ in Lemma~\ref{lem:slowconvergencesequence} with $(\max\{b_j, \eta + \varepsilon\})_{j \geq 0}$. 
\end{proof}

We are now ready to show that FDRS can converge arbitrarily slowly.
\begin{theorem}[Arbitrarily slow FDRS]\label{thm:arbitrarilyslow}
For every function $F : \vR_+ \rightarrow (0, 1)$ that strictly decreases to zero and satisfies $\{1/(j+1) \mid j \in \vN \backslash \{0\}\} \subseteq \range(F)$, there is a point $z^0 \in \ell_2^2(\vN)$ and two closed subspaces $U$ and $V$ with zero intersection, $U\cap V = \{0\}$, such that the FDRS sequence $(z^j)_{j \geq 0}$ generated with the functions $f := \chi_{U} + (a/2)\|\cdot\|^2$ and $g := (1/2)\|\cdot\|^2$ and parameters $\lambda_k \equiv 1$ and $\gamma = 1$ strongly converges to zero, but for all $k \geq 1$, we have $$\|z^k - z^\ast\| \geq e^{-1} F(k).$$
\end{theorem}
\begin{proof}
For all $i  \geq 0$, define $z_i^0 = (\|z_i\|^{-1}/(i+1))z_i$ with $z_i$ as in~\eqref{eq:eigenvalueFDRS}. Then $\|z_i^0\| = 1/(i+1)$ and $z_i^0$ is an eigenvector of $(\TFDRS)_i$ with eigenvalue $b_i := (a+c_i^2)/(a+1)$. Define the concatenated vector $z^0 := (z_i^0)_{i \geq 0}$. Note that $z^0 \in \cH$ because $\|z^0\|^2 = \sum_{i=0}^\infty 1/(i+1)^2 < \infty$. Thus, for all $k \geq 0$, we let $z^{k+1} := \TFDRS z^k$.  

Now, recall that $z^\ast= 0$. Thus, for all $n \geq 0$ and $k \geq 0$, we have 
\begin{align*}
\|z^{k+1} - z^\ast\|^2 = \|\TFDRS^{k+1} z^0 \|^2 = \sum_{i=0}^\infty b_i^{2(k+1)}\|z_i^0\|^2 = \sum_{i=0}^\infty \frac{b_i^{2(k+1)}}{(i+1)^2} \geq \frac{b_n^{2(k+1)}}{(n+1)^2}.
\end{align*}
Thus, $\|z^{k+1} - z^\ast\| \geq  b_n^{k+1}/(n+1)$.  Choose $b_n$ and the sequence $(n_j)_{j \geq 0}$ using Corollary~\ref{cor:slowconvergencesequence} with $\eta \in (a/(a+1), 1)$. Then solve $c_n = \sqrt{b_n(1+a) - a} > 0.$
\end{proof}

\begin{remark}
Theorems~\ref{thm:arbitrarilyslow} and~\ref{thm:strongconvexity} show that the sequence $(z^j)_{j \geq 0}$ can converge \emph{arbitrarily slowly} even if $(x_f^j)_{j \geq 0}$ and $(x_h^j)_{j \geq 0}$ converge with rate $o(1/\sqrt{k+1})$.
\end{remark}

The following theorem shows that $(x_f^j)_{j \geq 0}$ and $(x_h^j)_{j \geq 0}$ do not converge linearly. See Appendix~\ref{app:eq:nonlinearconvergence} for the proof.
\begin{theorem}\label{eq:nonlinearconvergence}
There exists a sequence $(c_i)_{i \geq 0}$ so that $(x_h^j)_{j \geq 0}$ and $(x_f^j)_{j \geq 0}$ converge strongly, but not linearly. In particular, for any $\alpha > 1/2$, there is an initial point $z^0 \in \cH$ so that for all $k \geq 1$,
\begin{align*}
\|x_h^k - x^\ast\|^2 \geq \frac{1}{(k+1)^{2\alpha}} && \text{and} && \|x_f^k - x^\ast\|^2 \geq  \frac{(a+1/2)^2}{(a+1)^2(k+1)^{2\alpha}}. 
\end{align*}
Thus, the nonergodic ``best" convergence rates in Part~\ref{thm:strongconvexity:part:nonergodic} of Theorem~\ref{thm:strongconvexity} are sharp.
\end{theorem}


\section{Primal-dual splittings}\label{sec:PD}
In this section, we reformulate FDRS as a primal-dual algorithm applied to the dual of the following problem: $\Min_{x \in V} f(x) + h(x)$.

\begin{lemma}[FDRS is a primal-dual algorithm]\label{lem:FDRSpd1}
Let $\tau := 1/\gamma$, and suppose that $(z^j)_{j \geq0}$ is generated by the FDRS algorithm with $\lambda_k \equiv 1$. For all $k \geq 0$, let  $y^{k} := -\tnabla \chi_{V}(x_h^k).$ Then for all $k \geq 0$, we have the recursive update rule:
\begin{align}\label{eq:pdalgsamesteps}
\begin{cases}
y^{k+1} &= P_{V^\perp}(y^k - \tau x_f^k); \\
x_f^{k+1} &= \prox_{\gamma f} \left( x_f^{k} - \gamma \nabla h (x_f^k) + \gamma (2y^{k+1} - y^k)\right).
\end{cases}
\end{align}
\end{lemma}
{\em Proof.}
Fix $k \geq 0$. By Lemma~\ref{lem:FDRSidentities}, $z^{k+1} = x_f^k -  \gamma y^k$\cut{(~\eqref{eq:zminusxasub})}, so $(-1/\gamma)z^{k+1} = y^k - \tau x_f^k$. Thus, the formula for $(y^j)_{j \geq 0}$ follows from $y^{k+1} = - \tnabla \chi_V(x_h^{k+1}) =  -(1/\gamma)P_{V^\perp}z^{k+1}$\cut{(~\eqref{eq:tadef})}.

Now observe that 
\begin{align*}
x_f^{k} = P_{V} x_f^{k} + P_{V^\perp} x_f^{k} \cut{\stackrel{\eqref{eq:zminusxasub}}{=}}&= P_{V} (z^{k+1} + \gamma y^{k}) + P_{V^\perp}( z^{k+1} + \gamma y^{k}) = x_h^{k+1} + \gamma(y^k - y^{k+1}).
\end{align*}
Furthermore, $\nabla h(x_f^{k}) = \nabla h(P_Vx_f^{k}) = \nabla h (P_V (z^{k+1} + \gamma y^{k})) = \nabla h (x_h^{k+1})$.  Thus,
\begin{align*}
x_f^{k+1} &\stackrel{\eqref{eq:FPRidentity}}{=} x_h^{k+1} -\gamma \left(\tnabla \chi_V(x_h^{k+1}) +  \nabla h(x_h^{k+1}) + \tnabla f(x_f^{k+1})\right)\\
&= \prox_{\gamma f} ( x_h^{k+1} - \gamma \nabla h(x_h^{k+1}) + \gamma y^{k+1}) \\
&= \prox_{\gamma f}(x_f^{k} - \gamma \nabla h (x_f^k) +  \gamma (2y^{k+1} - y^k)). \qquad\endproof
\end{align*}

The algorithm in~\eqref{eq:pdalgsamesteps} is the primal-dual forward-backward algorithm of V{\~u} and Condat~\cite{vu2013splitting,condat2013primal} applied to the following dual problem: $\Min_{x \in V^\perp} \; (f+h)^\ast(x)$ where $(f+h)^\ast(\cdot) = \sup_{x \in \cH} \dotp{x, \cdot} - (f + h)(x)$ is the Legendre-Fenchel transform of $f + h$~\cite[Definition 13.1]{bauschke2011convex}.
For convergence, \cite[Theorem 3.1]{vu2013splitting} requires $\gamma \tau < 1$ and  $2\beta_V > \left(\min\{1/\gamma, 1/\tau\}\left(1- \sqrt{\gamma \tau}\right)\right)^{-1}$ whereas FDRS requires $\gamma < 2\beta_V$ (and $\tau = 1/\gamma$).

Thus, the FDRS algorithm is a limiting case of V{\~u} and Condat's algorithm, much like the DRS algorithm~\cite{lions1979splitting} is a limiting case of Chambolle and Pock's primal-dual algorithm \cite{chambolle2011first}. In addition, the convergence rate analysis in Section~\ref{sec:objectiveconvergencerates} cannot be subsumed by the recent convergence rate analysis of the primal-dual gap of V{\~u} and Condat's algorithm \cite{davis2014convergencepd}, which only applies when $\gamma \tau < 1$. The original FDRS paper did not show this connection~\cite[Remark 6.3 (iii)]{briceno2012forward}. 


\section{Conclusion}
In this paper, we provided a comprehensive convergence rate analysis of the FDRS algorithm under general convexity, strong convexity, and Lip\-schitz differentiability assumptions.  In almost all cases, the derived convergence rates are shown to be sharp. In addition, we showed that the FDRS algorithm is the limiting case of a recently developed primal-dual forward-backward operator splitting algorithm and, thus, clarify how it relates to existing algorithms.  Future work on FDRS might evaluate the performance of the algorithm on realistic problems. 

\section*{Acknowledgement}
We thank Prof. Wotao Yin and the anonymous reviewers for helpful comments. We also thank the two anonymous referees for their insightful and detailed comments.

\appendix

\section{Performance improvement: $\beta_V$ versus $\beta$}\label{app:beta_vcompare}

\begin{figure}[htbp]
   \centering
   \includegraphics[width=0.8\textwidth]{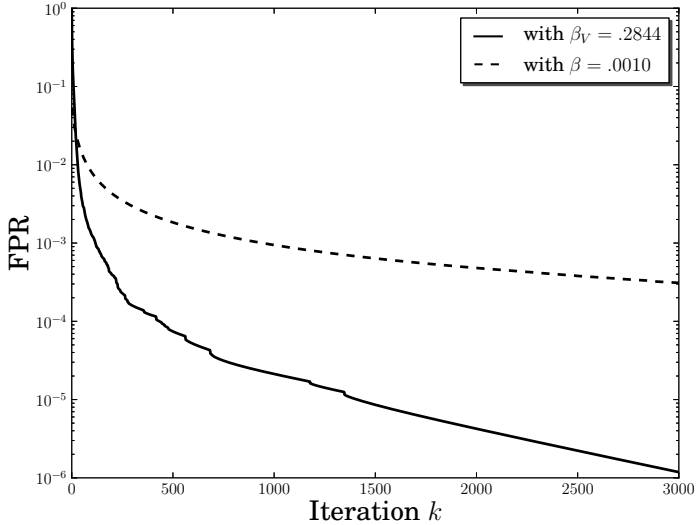}
   \caption{We plot the normalized FPR, $\|\TFDRS z^k - z^k\|/(1 + \|\TFDRS z^k\|)$, in a dual SVM example. See Appendix~\ref{app:beta_vcompare} for the details.}
   \label{fig:SVM}
\end{figure}

In this section, we briefly illustrate the benefits of using $\beta_V$ in place of $\beta$ on a Kernelized SVM problem, which is discussed in Section~\ref{sec:intro}; see \eqref{eq:cqprogramming} for notation. In Figure~\ref{fig:SVM} we plot the FPR associated to the FDRS algorithm applied to a 1000-dimensional quadratic program. To generate the quadratic program, we use  a random 1000-element subset of the the ``a7a" dataset (available from the LIBSVM website~\cite{CC01a}) denoted by $X = \{(x_1, y_1)^T, \cdots, (x_{1000}, y_{1000})^T\} \subseteq \vR^{ 123}$ where for each $i = 1, \cdots, 1000$, $x_i\in \vR^{122}$ is a data point and $y_i\in \{-1, 1\}$ is a class label. We use the matrix $Q \in \vR^{1000\times1000}$ with $i,j$ entry given by the formula $Q_{i,j} = y_iy_j\exp(-2^{-3}\|x_i - x_j\|^2)$ for $i, j \in \{1, \cdots, 1000\}$ (i.e., we use the radial basis function kernel). The matrix $A$ is the row vector $(y_1, \cdots, y_{1000}) \in \vR^{1\times 1000}$, and the set $C$ is the box $[0, 10]^{1000}\subseteq \vR^{1000}$. In this case, $P_V$ has rank $999$, but the maximal eigenvalue $(1/\beta_V \approx 3.5159)$ of $P_V\circ Q \circ P_V$ is approximately $275.8248$ times smaller than the maximal eigenvalue $(1/\beta \approx 969.7836)$ of $Q$. Figure~\ref{fig:SVM} shows that choosing $\gamma = 1.99 \beta_V$ results in a tremendous speedup. (In both examples, we chose $\lambda_k \equiv 1$.)

\section{Proofs of technical results}

\subsection{Proof of Proposition~\ref{prop:compogradientsum}}\label{app:prop:compogradientsum}
For the proof, we ask the reader to recall~\eqref{eq:lambdaboundwithepsilon}.

For all $k \geq 0$,  set 
\begin{align*}
p^k &:= \frac{1 - \alpha_1}{\alpha_1}\| (I_\cH - T_1)\circ T_2 (z^k) - (I_\cH - T_1)\circ T_2(z^\ast)\|^2 \\
&+\frac{1 - \alpha_2}{\alpha_2} \|(I_\cH - T_2)(z^k) - (I_\cH - T_2)(z^\ast)\|^2.
\end{align*}
By applying~\eqref{eq:avgdecrease} twice, we get $\|T_1\circ T_2 (z^k) - T_1\circ T_2 (z^\ast)\|^2 \leq \|z^k - z^\ast\|^2 - p^k.$

Part~\ref{prop:basicprox:part:wider} of Proposition~\ref{prop:basicprox} shows that $(T_1\circ T_2)_{\lambda_k}$ is $(\alpha_{1, 2}\lambda_k)$-averaged. Thus,
\begin{align*}
\|z^{k+1} - z^\ast\|^2 &\stackrel{\eqref{eq:avgdecrease}}{\leq} \|z^k - z^\ast\|^2 - \frac{\lambda_k(1-\lambda_k\alpha_{1, 2})}{\alpha_{1, 2}}\|T_1\circ T_2(z^k) - z^k\|^2.
\end{align*}
Therefore, $\sum_{i=0}^\infty\frac{\lambda_i(1-\alpha_{1, 2}\lambda_i)}{\alpha_{1, 2}}\|T_1\circ T_2(z^i) - z^i\|^2 \leq \|z^0 - z^\ast\|^2.$

By \cite[Corollary 2.14]{bauschke2011convex}, the following holds: for all $x, y \in \cH$ and all $\lambda \in \vR$, we have $\|\lambda x + (1-\lambda) y\|^2 = \lambda \|x\|^2 + (1-\lambda)\|y\|^2 - \lambda(1-\lambda)\|x - y\|^2.$
Therefore, we have
\begin{align*}
&\|z^{k+1} - z^\ast\|^2 \\
&= (1-\lambda_k)\|z^k - z^\ast\|^2 + \lambda_k\|T_1\circ T_2 (z^k) - T_1 \circ T_2(z^\ast)\|^2 - \lambda_k(1-\lambda_k)\|z^k - T_1\circ T_2(z^k)\|^2  \\
&\leq  \|z^k - z^\ast\|^2 - \lambda_k p^k +  \lambda_k(\lambda_k - 1)\|z^k - T_1\circ T_2(z^k)\|^2 \\
&\leq \|z^k - z^\ast\|^2 - \lambda_k p^k + \frac{\lambda_k(1-\alpha_{1, 2}\lambda_k)}{\alpha_{1, 2}\varepsilon}\|z^k - T_1\circ T_2(z^k)\|^2.
\end{align*}
Thus, take $k \rightarrow \infty$ in the following inequality to get the result:
\begin{align*}
&\sum_{i=0}^k \lambda_i\|(I_{\cH} - T_2)(z^i) - (I_{\cH} - T_2) (z^\ast)\|^2 \leq \frac{\alpha_2}{1-\alpha_2}\sum_{i=0}^k \lambda_i p^i \\
&\leq \frac{\alpha_2}{1-\alpha_2}\sum_{i=0}^k \left(\|z^i - z^\ast\|^2 - \|z^{i+1} - z^\ast\|^2 +  \frac{\lambda_i(1-\alpha_{1, 2}\lambda_i)}{\alpha_{1, 2}\varepsilon}\|z^i - T_1\circ T_2(z^i)\|^2\right) \\
&\leq \frac{\alpha_2(1+1/\varepsilon)\|z^0 - z^\ast\|^2}{1-\alpha_2}.\qquad \endproof
\end{align*}

\subsection{Proof of Lemma~\ref{lem:FDRSidentities}}\label{app:lem:FDRSidentities}
The identity for $x_h = z - \gamma \tnabla \chi_{V}(x_h)$ follows from Part~\ref{prop:basicprox:part:optprox} of Proposition~\ref{prop:basicprox}. Note that by the Moreau identity $P_{V^\perp} = I-P_{V}$, we have $\gamma \tnabla \chi_{V}(x_h) = P_{V^\perp} z$. Note that by definition, $\nabla h(z) = P_V \circ \nabla g \circ P_V(z) = P_V\circ \nabla g(x_h) = \nabla h(x_h)$ and $\nabla h(z) \in V$. Thus, we get the identity for $x_f$:
\begin{align*}
&\prox_{\gamma f}\circ \refl_{ \chi_V}\circ ( I_{\cH} - \gamma \nabla h)(z) = \refl_{\chi_V}\circ ( I_{\cH} - \gamma \nabla h)(z) - \gamma \tnabla f(x_f)  \\
&= x_h - \gamma \nabla h(z) - P_{V^\perp} z- \gamma \tnabla f(x_f) = x_h  - \gamma \left(\tnabla \chi_V(x_h) +  \nabla h(x_h) +  \tnabla f(x_f)\right).
\end{align*}

Finally, given the identity $(\TFDRS)_{\lambda}(z) - z = \lambda(\TFDRS(z) - z)$,~\eqref{eq:FPRidentity} will follow as soon as we show $\TFDRS(z)  = x_f + z - x_h = x_f + \gamma \tnabla \chi_{V}(x_h)$: 
\begin{align*}
 \left(\frac{1}{2}I_{\cH} + \frac{1}{2}\refl_{\gamma f} \circ \refl_{\chi_V}\right)(z - \gamma \nabla h(z)) &=\left(\prox_{\gamma f} \circ \refl_{\chi_{V}} + I_{\cH} - P_V\right)(z - \gamma \nabla h(z)) \\
&= x_f +  P_{V^\perp}(z - \gamma\nabla h(z)) = x_f + \gamma\tnabla \chi_{V}(x_h).\qquad\endproof
\end{align*}

\subsection{Proof of Lemma~\ref{lem:FDRSoptimality}}\label{app:lem:FDRSoptimality}
Let $x \in \zer(\partial f + \nabla h + \partial \chi_V)$. Choose subgradients $\tnabla f(x) \in \partial f(x)$ and $\tnabla \chi_V(x) \in \partial \chi_V(x) = V^\perp$ (by~\eqref{eq:normalcone}) such that $\tnabla f(x) + \nabla h(x) + \tnabla \chi_{V}(x)  = 0$ and set $z := x + \gamma\tnabla \chi_V(x)$. We claim that $z$ is a fixed-point of $\TFDRS$.  From Lemma~\ref{lem:FDRSidentities}, we get the points: $x_h := P_V(z) = x$ and $x_f := \prox_{\gamma f}\circ \refl_{ \chi_V}\circ ( I_\cH - \gamma \nabla h)(z).$ But $\tnabla \chi_V(x_h) +  \nabla h(x_h) \in -\partial f(x)$, and 
\begin{align*}
&\refl_{\chi_V}\circ ( I_\cH - \gamma \nabla h)(z) = P_V( z - \gamma \nabla h(z)) +  (P_V - I_{\cH}) ( z - \gamma \nabla h(z)) \\
&= x - \gamma \nabla h(x) - P_{V^\perp} z = x - \gamma \nabla h(x) - \gamma\tnabla \chi_V(x) = x + \gamma \tnabla  f(x).
\end{align*}
Therefore, $x_f = \prox_{\gamma f}(x + \gamma \tnabla  f(x)) = x = x_h$ (see Part~\ref{prop:basicprox:part:optprox} of Proposition~\ref{prop:basicprox}). Thus, by Lemma~\ref{lem:FDRSidentities}, $\TFDRS z = z + x_f - x_h= z$. We have proved the first inclusion.

On the other hand, suppose that $z \in \cH$ and $\TFDRS z = z$.  Then $x: = x_h = P_Vz$, and $0 = \TFDRS z -z = x_f - x_h = -\gamma \left(\tnabla \chi_V(x_h) +  \nabla h(x_h) +  \tnabla f(x_f)\right)$.  Because $x_f = x_h$, we get $x \in \zer(\partial f + \nabla  h + \partial \chi_V) $. \qquad \endproof

\subsection{Proof of Proposition~\ref{prop:FDRSupper}}\label{app:prop:FDRSupper}
In the following derivation, we use~\eqref{eq:fsbound} and~\eqref{eq:hsbound}, Lemma~\ref{lem:FDRSidentities},  the cosine rule, and the inclusion $\tnabla \chi_V(x_h) \in V^\perp$:
\begin{align*}
&2\gamma\lambda\left(f(x_f) + h(x_h) - f(x) - h(x) + S_f(x_f, x) + S_h(x_h, x)\right)\\
&\leq 2\gamma\lambda\left(\dotp{\tnabla f(x_f), x_f - x} + \dotp{ \nabla h(x_h), x_h - x} + \dotp{\tnabla \chi_V(x_h), x_h - x}\right) \\
&= 2\gamma \lambda\left(\dotp{\tnabla f(x_f) + \nabla h(x_h) + \tnabla \chi_V(x_h), x_f - x} + \dotp{ \nabla h(x_h) + \tnabla \chi_V(x_h), x_h - x_f}\right)  \\
&= 2 \dotp{z - z^+, x_f - x} +2 \dotp{\gamma \nabla h(x_h) + \gamma\tnabla \chi_V(x_h), z - z^+} \\
&= 2 \dotp{z - z^+, x_f + \gamma \tnabla \chi_V(x_h) - x} + 2\gamma\dotp{\nabla h(x_h), z - z^+} \\
&= 2 \dotp{z - z^+,\TFDRS z - x} + 2\gamma\dotp{\nabla h(x_h), z - z^+} \\
&= 2 \dotp{z - z^+,z - x} + \frac{2}{\lambda}\dotp{z - z^+, z^+ - z} + 2\gamma\dotp{\nabla h(x_h), z - z^+}\\
&\stackrel{\eqref{eq:cosinerule}}{=} \|z - x\|^2 - \|z^+ - x\|^2 + \left(1-\frac{2}{\lambda}\right) \| z^+ - z\|^2 + 2\gamma \dotp{ \nabla h(x_h), z - z^+}.\qquad \endproof
\end{align*}

\subsection{Proof of Proposition~\ref{prop:FDRSlower}}\label{app:prop:FDRSlower}
By~\eqref{eq:fsbound} and~\eqref{eq:hsbound} and because $\tnabla \chi_V(x^\ast) \in V^\perp$, we have
\begin{align*}
f(x_f) + h(x_h) - f(x^\ast) - g(x^\ast) &\geq \dotp{x_h - x^\ast, \tnabla f(x^\ast) + \nabla h(x^\ast) + \tnabla \chi_V(x^\ast)}  \\
&+ \dotp{x_f - x_h, \tnabla f(x^\ast)} + S_f(x_f, x^\ast) + S_h(x_h, x^\ast)\\
&=  \dotp{x_f - x_h, \tnabla f(x^\ast)} + S_f(x_f, x^\ast) + S_h(x_h, x^\ast). \qquad \endproof
\end{align*}

\subsection{Proof of Corollary~\ref{cor:strongconvexfundamental}}\label{app:cor:strongconvexfundamental} By~\eqref{eq:cosinerule}, we have $\|z - x^\ast\|^2 - \|z^+ - x^\ast\|^2 = \|z - z^\ast\|^2 - \|z ^+ - z^\ast\|^2 + 2\dotp{z - z^+, z^\ast - x^\ast}.$
Therefore, by Proposition~\ref{prop:FDRSupper},
\begin{align*}
2\gamma\lambda&\left(f(x_f) + h(x_h) - f(x^\ast) - h(x^\ast) + S_f(x_f, x^\ast) + S_h(x_h, x^\ast)\right) \\
&\leq \|z - z^\ast\|^2 - \|z^+ - z^\ast\|^2 + 2\dotp{z - z^+, z^\ast - x^\ast} \\
&+ \left(1-\frac{2}{\lambda}\right) \| z^+ - z\|^2 + 2\gamma \dotp{ \nabla h(x_h), z - z^+}. \numberthis \label{eq:FDRSupper2}
\end{align*}

Equation~\eqref{eq:strongconvexupper} now follows from~\eqref{eq:FDRSupper2} and~\eqref{eq:FDRSlower}:
\begin{align*}
&4\gamma \lambda(S_f(x_f, x^\ast) + S_h(x_h, x^\ast)) \stackrel{\eqref{eq:FDRSlower}}{\leq} - 2\gamma \lambda\dotp{x_f - x_h, \tnabla f(x^\ast)} \\
&+ 2\gamma\lambda(f(x_f) + h(x_h) - f(x^\ast) - h(x^\ast) + S_f(x_f, x^\ast) + S_h(x_h, x^\ast)) \\
&\stackrel{\eqref{eq:FDRSupper2}}\leq \|z - z^\ast\|^2 - \|z^{+} - z^\ast\|^2 + 2\dotp{z - z^{+}, z^\ast - x^\ast} - 2\gamma\lambda\dotp{x_f - x_h, \tnabla f(x^\ast)}\\
&+ \left(1-\frac{2}{\lambda}\right) \| z^{+} - z\|^2 + 2\gamma \dotp{ \nabla h(x_h), z - z^{+}} \\
&\stackrel{\eqref{eq:FPRidentity}}{=}\|z - z^\ast\|^2 - \|z^{+} - z^\ast\|^2 + \left(1-\frac{2}{\lambda}\right) \| z^{+} - z\|^2 + 2\gamma \dotp{ \nabla h(x_h) - \nabla h(x^\ast), z - z^{+}}.\qquad \endproof
\end{align*} 

\subsection{Proof of Theorem~\ref{thm:strongconvexity}}\label{app:thm:strongconvexity}

Let $\eta_k = 2/\lambda_k - 1$.  By~\eqref{eq:ergodicauxiliarybound}, we have
\begin{align}
2\gamma \dotp{ \nabla h(x_h^k) - \nabla h(x^\ast), z^k - z^{k+1}} &\leq \frac{\gamma^2}{\eta_k} \|\nabla h(x_h^k) - \nabla h(x^\ast)\|^2 + \eta_k\|z^k - z^{k+1}\|^2. \label{eq:crosstermbound}
\end{align}
Hence, for all $k \geq 0$, we have (using $1/\eta_k \leq \lambda_k/\varepsilon^2$ as in~\eqref{eq:ergodicauxiliarybound} and~\eqref{eq:lambdaboundwithepsilon})
\begin{align*}
&4\gamma \underline{\lambda}\sum_{i=0}^k (S_f(x_f^i, x^\ast) + S_h(x_h^i, x^\ast)) \leq \sum_{i=0}^k 4\gamma \lambda_i(S_f(x_f^i, x^\ast) + S_h(x_h^i, x^\ast)) \\
&\stackrel{\eqref{eq:strongconvexupper}}{\leq} \sum_{i=0}^k \biggl( \|z^i - z^\ast\|^2 - \|z^{i+1} - z^\ast\|^2 - \eta_i \| z^{i+1} - z^i\|^2 \\
&\quad\quad\quad\quad+ 2\gamma \dotp{ \nabla h(x_h^i) - \nabla h(x^\ast), z^i - z^{i+1}}\biggr) \\
&\stackrel{\eqref{eq:crosstermbound}}{\leq} \sum_{i=0}^k \left( \|z^i - z^\ast\|^2 - \|z^{i+1} - z^\ast\|^2 + (\gamma^2\lambda_i/\varepsilon^2)\|\nabla h(x_h^i) - \nabla h(x^\ast)\|^2\right)\\
&\stackrel{\eqref{eq:gradientsum}}{\leq} \|z^0 - z^\ast\|^2 - \|z^{k+1} - z^\ast\|^2 + \frac{(1+\varepsilon)\gamma}{\varepsilon^3(2\beta_V-\gamma)}\|z^0 - z^\ast\|^2.
\end{align*}
The ``best" convergence rates now follow by taking $k \rightarrow \infty$ and using \cite[Lemma 3]{davis2014convergence}.  In addition, we apply Jensen's inequality to $\|\cdot\|^2$ in the first term to get
\begin{align*}
\frac{\mu_f}{2}\|\overline{x}_f^k - x^\ast\|^2 + \frac{\mu_h}{2} \|\overline{x}_h^k - x^\ast\|^2  \cut{&\leq \frac{1}{\Lambda_k} \sum_{i=0}^k \lambda_i(S_f(x_f^i, x^\ast) + S_h(x_h^i, x^\ast)) \\}
&\leq \frac{\left(1 + \frac{(1+\varepsilon)\gamma}{\varepsilon^3(2\beta_V-\gamma)} \right)\|z^0 - z^\ast\|^2}{4\gamma\Lambda_k}.
\end{align*}

We now fix $k \geq 0$. For all $\lambda > 0$, define $z_{\lambda} := (\TFDRS)_\lambda(z^k)$.\cut{ Again, we note that $z_\lambda - z^\ast = \lambda(\TFDRS(z^k) - z^k)$ and, hence, we can bound the size of $z_{\lambda} - z^k$ using~\eqref{thm:FDRSfacts:part:convergenceFPR:eq}.} Observe that $S_f(x_f^k, x^\ast)$ and $S_h(x_h^k, x^\ast)$ do not depend on the value of $\lambda_k$.  Therefore, we use~\eqref{eq:strongconvexupper} to get
\begin{align*}
S_f(x_f^k, x^\ast) + S_h(x_h^k, x^\ast) &\leq \inf_{\lambda \in [0, 1/\alphaFDRSV)}\frac{1}{4\gamma\lambda} \biggl(2\gamma \dotp{\nabla h(x_h^k) - \nabla h(x^\ast), z^k - z_\lambda} \\
&+ \|z^k - z^\ast\|^2 - \|z_\lambda - z^\ast\|^2 + \left(1-\frac{2}{\lambda}\right) \| z_\lambda - z^k\|^2\biggl)  \\
&\stackrel{\eqref{eq:cosinerule}}{=} \inf_{\lambda \in [0, 1/\alphaFDRSV)}\frac{1}{4\gamma\lambda} \biggl(2\gamma \dotp{\nabla h(x_h^k) - \nabla h(x^\ast), z^k - z_\lambda} \\
&+ 2\dotp{z_\lambda - z^\ast, z^k - z_\lambda} + 2\left(1 - \frac{1}{\lambda}\right)\| z_\lambda - z^k\|^2\biggr)\\
&\leq  \frac{1}{4\gamma}\left(2\dotp{z_1 - z^\ast, z^k - z_1} + \frac{2\gamma}{\beta_V}\|z^k - z^\ast\| \|z_1 - z^k\|\right) \numberthis \label{eq:littleoupper2} \\
&\stackrel{\eqref{thm:FDRSfacts:part:convergenceFPR:eq}}{\leq} \frac{(1+ {\gamma}/{\beta_V})\|z^0 - z^\ast\|^2}{2\gamma\sqrt{\underline{\tau}(k+1)}}
\end{align*}
where~\eqref{eq:littleoupper2} uses the ($1/\beta_V$)-Lipschitz continuity of $\nabla h$ and the identity $\nabla h(x_h^k) - \nabla h(x^\ast) = \nabla h(z^k) - \nabla h(z^\ast)$, and the last line uses the Fej\'er property $\|z_1 - z^\ast\| \leq \|z^k - z^\ast\| \leq \|z^0 - z^\ast\|$ (see Part~\ref{thm:FDRSfacts:part:fejer} of Theorem~\ref{thm:FDRSfacts}). The $o(1/\sqrt{k+1})$ rates follow from~\eqref{eq:littleoupper2} and the corresponding rates for the FPR in~\eqref{thm:FDRSfacts:part:convergenceFPR:eq}.  \endproof

\subsection{Proof of Proposition~\ref{eq:newfundamentalinequality}}\label{app:eq:newfundamentalinequality}

Because $\nabla f$ is ($1/\beta_f$)-Lipschitz, we have 
\begin{align}
f(x_h) &\leq f(x_f) + \dotp{x_h - x_f, \nabla f(x_f)} + \frac{1}{2\beta_f}\|x_h - x_f\|^2; \label{eq:descenttheorem2}\\
S_f(x_f, x^\ast) &\stackrel{\eqref{eq:Sf}}{\geq} \frac{\beta_f}{2} \|\nabla f(x_f) - \nabla f(x^\ast)\|^2\label{eq:Slowerbound}.
\end{align} 
where the first inequality follows from~\cite[Theorem 18.15(iii)]{bauschke2011convex}. By applying the identity $z^\ast - x^\ast = \gamma \tnabla \chi_V(x^\ast) =- \gamma \nabla f(x^\ast) - \gamma\nabla h(x^\ast)$, the cosine rule~\eqref{eq:cosinerule}, and the identity $z - z^+ = \lambda(x_h - x_f)$ (see~\eqref{eq:FPRidentity}) multiple times, we have
\begin{align*}
&2\dotp{z - z^+, z^\ast - x^\ast} + 2\gamma \lambda\dotp{ x_h - x_f, \nabla f(x_f)} = 2\lambda\dotp{x_h - x_f, \gamma \tnabla \chi_V(x^\ast) + \gamma\nabla f(x_f)} \\
&= 2\lambda \dotp{\gamma \tnabla \chi_V(x_h) + \gamma \nabla h(x_h) + \gamma \nabla f(x_f) , \gamma \nabla f(x_f) - \gamma \nabla f(x^\ast)} - 2\dotp{z-z^+, \gamma \nabla h(x^\ast)} \\
&= \lambda \biggl(\|\gamma\nabla f(x_f) - \gamma \nabla f(x^\ast)\|^2 + \|x_h - x_f\|^2 \\
&- \gamma^2\|\tnabla \chi_V(x_h) + \nabla h(x_h) -  \tnabla \chi_V(x^\ast) -  \nabla h(x^\ast)\|^2\biggr) - 2\dotp{z-z^+, \gamma \nabla h(x^\ast)}. \numberthis \label{eq:preboundLipschitz}
\end{align*}
By~\eqref{eq:FPRidentity} (i.e., $z - z^+ = \lambda(x_h - x_f)$), we have
\begin{align*}
\left(1 - \frac{2}{\lambda}\right) \|z - z^+\|^2 + \lambda \left(\frac{\gamma}{\beta_f} + 1\right) \|x_h - x_f\|^2 &= \left(1 + \frac{\gamma - \beta_f}{\beta_f\lambda}\right)\|z-z^+\|^2.
\end{align*}
Therefore, 
\begin{align*}
&2\gamma \lambda (f(x_h) + h(x_h) - f(x^\ast) - h(x^\ast)) \\
&\stackrel{\eqref{eq:descenttheorem2}}{\leq} 2\gamma \lambda(f(x_f) + h(x_h) - f(x^\ast) - h(x^\ast)) +  2\gamma \lambda\dotp{ x_h - x_f, \nabla f(x_f)} + \frac{\gamma\lambda}{\beta_f}\|x_h - x_f\|^2 \\
&\stackrel{\eqref{eq:FDRSupper2}}{\leq} \|z - z^\ast\|^2 - \|z^+ - z^\ast\|^2 + 2\dotp{z - z^+, z^\ast - x^\ast} + 2\gamma \lambda\dotp{ x_h - x_f, \nabla f(x_f)}\\
&+ \left(1-\frac{2}{\lambda}\right) \| z^+ - z\|^2 + 2\gamma \dotp{ \nabla h(x_h), z - z^+} + \frac{\gamma\lambda}{\beta_f}\|x_h - x_f\|^2 - 2\gamma \lambda S_f(x_f, x^\ast)  \\
&\stackrel{\eqref{eq:preboundLipschitz}}{\leq} \|z - z^\ast\|^2 - \|z^+ - z^\ast\|^2 +\left(1 - \frac{2}{\lambda}\right) \|z - z^+\|^2 + \lambda \left(\frac{\gamma}{\beta_f} + 1\right) \|x_h - x_f\|^2 \\
&+ \lambda \|\gamma\nabla f(x_f) - \gamma\nabla f(x^\ast)\|^2 + 2\gamma \dotp{ \nabla h(x_h) - \nabla h(x^\ast), z - z^+}- 2\gamma \lambda S_f(x_f, x^\ast) \\
&\stackrel{\eqref{eq:Slowerbound}}{\leq}\|z - z^\ast\|^2 - \|z^+ - z^\ast\|^2 +   \left(1 + \frac{\gamma - \beta_f}{\beta_f\lambda}\right)\|z-z^+\|^2 \\
&+ 2\gamma \dotp{ \nabla h(x_h) - \nabla h(x^\ast), z - z^+} +  \gamma \lambda (\gamma - \beta_f)\|\nabla f(x_f) - \nabla f(x^\ast)\|^2. \numberthis\label{eq:lipschitzdiffbound}
\end{align*}
If $\gamma \leq \beta_f$, then we can drop the last term.  If $\gamma > \beta_f$, then use~\eqref{eq:strongconvexupper} to get
\begin{align*}
\gamma \lambda (\gamma - \beta_f)\|\nabla f(x_f) - \nabla f(x^\ast)\|^2 &\leq \frac{ \gamma - \beta_f}{2\beta_f}\biggr(2\gamma \dotp{ \nabla h(x_h) - \nabla h(x^\ast), z - z^{+}}\\
&+ \|z - z^\ast\|^2 - \|z^+ - z^\ast\|^2 + \left(1-\frac{2}{\lambda}\right) \| z^+ - z\|^2 \biggr)
\end{align*}
The result follows by~\eqref{eq:lipschitzdiffbound} and $$\left(1 + \frac{\gamma - \beta_f}{\beta_f\lambda}\right)\|z-z^+\|^2 + \frac{ \gamma - \beta_f}{2\beta_f} \left(1-\frac{2}{\lambda}\right) \| z - z^+\|^2 = \left(1 + \frac{\gamma - \beta_f}{2\beta_f}\right)\|z - z^+\|^2. \qquad \endproof$$

\subsection{Proof of Theorem~\ref{eq:nonlinearconvergence}}\label{app:eq:nonlinearconvergence}

For all $i \geq 0$, let $c_i := (i/(i+1))^{1/2}$. Let $\kappa_a := (1/2) + 2(a+1)^2$, and let $z^0 := \sqrt{2\alpha \kappa_a}e^{(1/(a+1))} \times \left( (\|z_i\|^{-1}/(i+1)^\alpha)z_i\right)_{i \geq 0}.$ Then $\|z^0\|^2 = 2\alpha\kappa_a e^{2/(a+1)}\sum_{i=0}^\infty (1/(i+1)^{2\alpha}) < \infty$ and, hence, $z^0 \in \cH$.  Now for all $i \geq 1$, we have
\begin{align*}
\frac{\|z_i\|^2 (a+c_i^2)^2}{c_i^2} \stackrel{\eqref{eq:eigenvectornorm}}{=} (1-c_i^2) + \frac{ (a+c_i^2)^2}{c_i^2} &\leq \kappa_a \numberthis\label{eq:kappa}
\end{align*}
because $c_i^2 \in [1/2, 1)$. In addition, for all $i \geq 1$, we have
\begin{align*}
\|(P_V)_iz_i^0\|^2 &= \frac{2\alpha\kappa_a e^{2/(a+1)}}{\|z_i\|^2(i+1)^{2\alpha}} \|(P_V)_i z_i\|^2  \stackrel{\eqref{eq:eigenvectornorm}}{=} \frac{2\alpha \kappa_a e^{2/(a+1)}c_i^2(1-c_i^2)}{\|z_i\|^2(a + c_i^2)^2(i+1)^{2\alpha}} \\
&= \frac{2\alpha\kappa_a e^{2/(a+1)}c_i^2}{\|z_i\|^2(a + c_i^2)^2(i+1)^{1+2\alpha}} \stackrel{\eqref{eq:kappa}}{\geq}  \frac{2\alpha e^{2/(a+1)}}{(i+1)^{1+2\alpha}}
\end{align*}
where the third equality follows because $1-c_i^2 = 1 - i/(i+1) = 1/(i+1)$.  

Now, for all $k \geq 0$, let $z^{k+1} := \TFDRS z^k$. Again, for all $i \geq 0$, let $b_i := (a + c_i^2)/(a + 1) = 1 - (i+1)^{-1}(a+1)^{-1}$ be the eigenvalue of $(\TFDRS)_i$ associated to $z_i$. Note that $b_i^{2k} \geq e^{-2/(1+a)}$ whenever $i \geq k \geq 0$ (hint: use the bound $e^{-1/(a+1)} \leq ( 1 - (i+1)^{-1}(a+1)^{-1})^{i} = b_{i}^i$, and note that $b_i^{2k}$ is increasing in $i$ for fixed $k$). Therefore, for all $k \geq 1$, we have 
\begin{align*}
\|x_h^k - x^\ast\|^2 &= \|P_V\TFDRS^k z^0 \|^2 = \sum_{i=0}^\infty b_i^{2k}\|(P_V)_iz_i^0\|^2 \geq \sum_{i=k}^\infty b_i^{2k}\frac{2\alpha e^{2/(a+1)}}{(i+1)^{1+2\alpha}} \\
&\geq\sum_{i=k}^\infty \frac{2\alpha}{(i+1)^{1+2\alpha}} \geq \frac{1}{(k+1)^{2\alpha}}. \numberthis \label{eq:slowboundforxh}
\end{align*}
where we use $x^\ast = 0$ and the lower integral approximation of the sum. 

Now we prove the bound for $(x_f^j)_{j \geq 0}$. For all $k \geq 0$,  $x_f^k = \TFDRS z^k - \gamma \tnabla \chi_V(x_h^k) = \TFDRS z^k - P_{V^\perp} z^k = (\TFDRS - P_{V^\perp})\TFDRS^k z^0$ (see~\eqref{lem:FDRSidentities}).  In addition, for all $i \geq 0$,
\begin{align*}
(\TFDRS - P_{V^\perp})_i = \frac{1}{(a + 1)}\begin{bmatrix} 0 & -\cos(\theta_i)\sin(\theta_i) \\ 0 & \cos^2(\theta_i) + a - (a+1) \end{bmatrix} = -\frac{\sin(\theta_i)}{(a+1)} \begin{bmatrix} 0 & \cos(\theta_i) \\ 0 & \sin(\theta_i) \end{bmatrix}.
\end{align*}
Thus, for all $i \geq 0$, we have 
\begin{align*}
\|(\TFDRS - P_{V^\perp})_i z_i^0\|^2 &= \frac{2\alpha\kappa_a e^{2/(a+1)}\sin^2(\theta_i)(\cos^2(\theta_i) + \sin^2(\theta_i))}{\|z_i\|^2(a+1)^2(i+1)^{2\alpha}} = \frac{2\alpha\kappa_a e^{2/(a + 1)}(1-c_i^2)}{\|z_i\|^2(a+1)^2(1+i)^{2\alpha}} \\
&\stackrel{\eqref{eq:kappa}}{\geq} \frac{2\alpha e^{2/(a+1)}(a+c_i^2)^2}{c_i^2(a+1)^2(1+i)^{1 + 2\alpha}}.
\end{align*}
where the last inequality follows because $1-c_i^2 = 1 - i/(i+1) = 1/(i+1)$ and $\kappa_a/\|z_i\|^2 \geq (a+c_i^2)^2/c_i^2$.  Note that for all $i \geq 1$, we have $(a+c_i^2)^2/c_i^2 \geq (a+1/2)^2$ because $c_i^2 \in [1/2, 1)$.  Therefore, for all $k \geq 1$, we have
\begin{align*}
\|x_f^k - x^\ast\|^2 = \|(\TFDRS - P_{V^\perp})\TFDRS^k z^0\|^2 
&\geq \sum_{i=k}^\infty b_i^{2k}\frac{2\alpha e^{2/(a+1)}(a+c_i^2)^2}{c_i^2(a+1)^2(1+i)^{1 + 2\alpha}} \\
&\geq \frac{(a+1/2)^2}{(a+1)^2(k+1)^{2\alpha}}
\end{align*}
where we use similar arguments to those used in~\eqref{eq:slowboundforxh}. \qquad \endproof


\bibliographystyle{siam}
\bibliography{099229Bibliography}


\end{document}